\documentclass[11pt,letterpaper]{amsart}

\usepackage{mathrsfs}
\usepackage{amssymb}
\usepackage{amsmath}
\usepackage{amsthm}
\usepackage{amsfonts}
\usepackage{extarrows}
\usepackage{xy}
\xyoption{all}

\usepackage{color}
\usepackage[colorlinks, linkcolor=red, anchorcolor=green, citecolor=blue]{hyperref}
\usepackage[pagewise]{lineno}%\linenumbers

\allowdisplaybreaks

\renewcommand\hat{\widehat}
\renewcommand\tilde{\widetilde}

\newtheorem{theorem}{Theorem}[section]
\newtheorem{lemma}[theorem]{Lemma}
\newtheorem{corollary}[theorem]{Corollary}

\newtheorem{example}[theorem]{Example}
\theoremstyle{definition}
\newtheorem{remark}[theorem]{Remark}
\newtheorem{definition}[theorem]{Definition}

\newtheorem {open problem posed by Peetre} [theorem] {Open problem posed by Peetre}
\newtheorem {open problem} [theorem] {Open problem}
\numberwithin{equation}{section}

\numberwithin{equation}{section}

\numberwithin{equation}{section}

\begin{document}
\UseRawInputEncoding
\arraycolsep=1pt
\title[Interpolation spaces of Besov hierarchical spaces]{\large\bf
Interpolation spaces of Besov hierarchical spaces and non-linearities defined by vertex K functional and grid topology}

\let\thefootnote\relax\footnotetext{\hspace{-0.15cm}
{\it 2020 Mathematics Subject Classification}. {46B70, 46E35, 42B35.}
\endgraf{\it Key words and phrases.}
Real interpolation; Besov hierarchical space; Wavelet basis and grid;  nonlinear functional structure;
nonlinear topology structure.
\endgraf $^\ast$\,Corresponding author
}

\author{
% Dunyan Yan, Haibo Yang $^{\ast}$ and Qixiang  Yang\\
%Haibo Yang, Qixiang  Yang $^{\ast}$ and Dunyan Yan \\
Qixiang  Yang, Haibo Yang $^{\ast}$
}

%\date{  }
\maketitle

\vspace{-0.8cm}

\begin{center}
\begin{minipage}{13cm}\small
{\noindent{\bf Abstract.}
In 1967, Peetre proposed to give a precise description of the real interpolation space for
Besov hierarchical spaces $l^{s,q}(A)$.
In 1974, Cwikel proved that
the Lions-Peetre formula for $(l^{q_0}(A_0), l^{q_1}(A_1))_{\theta,r}$
have no reasonable generalization for any $r\neq q$.
In this paper,
we apply wavelets to transform the study of real interpolation space into
{\bf the study of nonlinear functional structure and nonlinear topological structure}.
We solve completely Peetre's longstanding open problem.
}

\end{minipage}
\end{center}

%\tableofcontents

%%%%%%%%%%%%%%%%%%%%%%%%%%%%%%%%%%%%%%%%%%%%%%%%%%%%%%%%%%%%%%%%%%

%%%%%%%%%%%%%%%%%%%%%% section 1 %%%%%%%%%%%%%%%%%%%%%%%%%%%%%%%%%%
%%%%%%%%%%%%%%%%%%%%%%%%%%%%%%%%%%%%%%%%%%%%%%%%%%%%%%%%%%%%%%%%%%%%%
\section{Main theorem}\label{s1}

\subsection{Peetre's long standing open problem}
%\hskip\parindent
Besov spaces $\dot{B}^{s,q}_{p}=\dot{l}^{s,q}(L^{p})$ and
$B^{s,q}_{p}=l^{s,q}(L^{p})$ have been studied heavily at 1950s.
If we replace $L^{p}$ with pseudo-Banach space $A$, then we get Besov hierarchical spaces $\dot{l}^{s,q}(A)$.
Peetre's book \cite{Peetre} introduced new thoughts on Besov space
and has received much attention in the fields of differential equations and harmonic analysis.
See also Bergh-L\"ofstr\"om's book \cite{BL}.
Peetre wrote on page 42 of \cite{Peetre}:
The abstract theory of interpolation spaces was created around 1960 by Lions, Gagliardo, Calderon, Krein and others.
The complex spaces are studied in Calderon \cite{Cal}. The real spaces are studied in Lions-Peetre \cite{LP}
and Peetre \cite{Peetre}.
Many results on interpolation have been highlighted by leading journals.
See \cite{ B, CZ, C,  D, FRS, FS, Hunt1, Litt, Lorentz, R, SW, SZ}.

Given $0<\theta<1, s_0,s_1\in \mathbb{R},0<q_0,q_1,r\leq \infty$ and $A_0,A_1$ Banach spaces.
For $\frac{1}{q}= \frac{1-\theta}{q_0} + \frac{\theta}{q_1}$,
by applying Lions-Peetre's Theorem \cite{LP},
the interpolation spaces
$(l^{s_0,q_0}(A_0), l^{s_1,q_1}(A_1))_{\theta,r}$
are known for $r=q$.
\begin{open problem} \label{con:1}
Given $0<\theta<1, s_0,s_1\in \mathbb{R},0<q_0,q_1,r\leq \infty$ and $A_0,A_1$ Banach spaces.
For $\frac{1}{q}= \frac{1-\theta}{q_0} + \frac{\theta}{q_1}$,
Peetre proposed to find out the expression of
$(l^{s_0,q_0}(A_0),  l^{s_1,q_1}(A_1))_{\theta, r}$ for $r\neq q$.
See his book \cite{Peetre67book} in 1967 or
page 104 of his English translation \cite{Peetre} in 1976.
\end{open problem}
Apart from Lorentz space, there's been a lot of research on Hardy spaces and so on.
But for the general indicators, there is no systematic study.
For $(l^{q_0}(A_0), l^{q_1}(A_1))_{\theta,r}$, in 1974, Cwikel \cite{C} proved that
{\bf the Lions-Peetre formula
have no reasonable generalization for any $r\neq q$ where $\frac{1}{q}= \frac{1-\theta}{q_0} + \frac{\theta}{q_1}$.}
In this paper, we develop new skills to solve completely this Open problem \ref{con:1}.
The K functional $K(t, f, \dot{l}^{s_0,q_0}(A_0), \dot{l}^{s_1,q_1}(A_1))$
in the case of homogeneous space is considered.
For the case of a non-homogeneous space,
its K functional $K(t, f, l^{s_0,q_0}(A_0), l^{s_1,q_1}(A_1))$ can be similarly obtained.
Our main theorem is to work out K functionals for Besov hierarchical spaces
through vertex K functional and grid topology structure
by analyzing the nonlinear structure of K functional.
The idea is that for a distribution $f$,
we get a set of sequences $\{f_{j,\gamma}\}_{j\in \mathbb{Z}, \gamma\in \dot{\Gamma}_{j}}$
uniquely by wavelet decomposition.
The underlying assumption of Peetre's problem is that
norms such as $\| \{f_{j,\gamma}\}_{\gamma\in \dot{\Gamma}_{j}}\| _{A_{i} (\dot{\Gamma}_{j})}$,
various K functional $K(t,\{f_{j,\gamma}\}_{\gamma\in \dot{\Gamma}_{j}}, A_0, A_1)$
and quantities associated with such quantities are known.
We start with these quantities defined in terms of wavelet coefficients to give
the K functional $K(t, f, \dot{l}^{s_0,q_0}(A_0), \dot{l}^{s_1,q_1}(A_1)).$
We express the K functional $K(t,f, \dot{l}^{s_0,q_0}(A_0),  \dot{l}^{s_1,q_1}(A_1))$
as some composite function of some known K functional and some K functional related to $(A_0,  A_1)$.
Our results here cover all the known results
related to the expression of the real interpolation of Besov type spaces.

We replace the Banach space hypothesis in Open problem \ref{con:1} with more general pseudo-Banach spaces.
We divide into four different index relations to describe the corresponding interpolation space for pseudo-Banach spaces.
\begin{theorem}\label{th:main}
Given $0<\theta<1, 0< q_0,q_1,r \leq\infty$, $s_0,s_1\in \mathbb{R}$.
If pseudo-Banach spaces $A_0, A_1$ satisfy \eqref{eq:absolute},
then
$$(\dot{l}^{s_0,q_0}(A_0), \dot{l}^{s_1,q_1}(A_1))_{\theta,r}
=\{\int^{\infty}_{0} [ t^{-\theta} K(t, f, A_0, A_1)] ^{\eta}
\frac{dt}{t} \}^{\frac{1}{\eta}},$$
where the K functional
$K(t, f, \dot{l}^{s_0,q_0}(A_0), \dot{l}^{s_1,q_1}(A_1))$ can be formulated in four cases as follows:

(i) If $A_0=A_1=A$, then
$$K(t, f, \dot{l}^{s_0,q_0}(A_0), \dot{l}^{s_1,q_1}(A_1))= K(t,F, \dot{l}^{s_0,q_0}(\mathbb{Z}), \dot{l}^{s_1,q_1}(\mathbb{Z})),$$
where $F=(F_j)_{j\in \mathbb{Z}}$, $F_{j}=\|f_{j,\gamma}\|_{A(\dot{\Gamma}_{j})}$
and $K(t,F, \dot{l}^{s_0,q_0}(\mathbb{Z}), \dot{l}^{s_1,q_1}(\mathbb{Z}))$ is known.
See Remark \ref{re:lsq} and Theorem \ref{th:5.1}.

(ii) If $0<q_0=q_1\leq \infty$, then $X_{j}(t) $ and $Y_{j}(t)$
determined by K functional
$K_{V}(t, \{f_{j,\gamma}\}_{\gamma\in \dot{\Gamma}_{j}}, A_0, A_1)$ is known
and $$K(t, f, \dot{l}^{s_0,q_0}(A_0), \dot{l}^{s_1,q_1}(A_1))$$
can be obtained by using $X_{j}(t) $ and $Y_{j}(t)$
in Theorem \ref{th:5.2}.

(iii) If $0<q_0\neq q_1<\infty$, then
$K_{\infty}(t, \{f_{j,\gamma}\}_{\gamma\in \dot{\Gamma}_{j}}, A_0, A_1)$ is known.
Theorem \ref{th:11.4} tells us that, if $q_0<q_1$ and $g^{0}_{j}=(2^{js_0}|f_{j,\gamma}|)_{\gamma\in \Gamma_{j}}$, then
$$\begin{array}{rcl}
K_{\infty}(t,f, \dot{l}^{s_0,q_0}(A_0), \dot{l}^{s_1, q_1}(A_1))
&=& \{\sum\limits_{j\in \mathbb{Z}} s2^{j\tilde{s}^{0}} K_{\infty} (\tau, g^{0}_j, A_1(\dot{\Gamma}_j), A_0(\dot{\Gamma}_j))^{q_1} \}^{\frac{1}{q_0}}.
\end{array}$$
If $q_1<q_0$ and $g^{1}_{j}=(2^{js_1}|f_{j,\gamma}|)_{\gamma\in \Gamma_{j}}$, then
$$\begin{array}{rcl}
K_{\infty}(t,f, \dot{l}^{s_0,q_0}(A_0), \dot{l}^{s_1, q_1}(A_1))&=&
t \{\sum\limits_{j\in \mathbb{Z}} s^{2} 2^{j\tilde{s}^{1} } K_{\infty} (\tau, g_j^{1}, A_0(\dot{\Gamma}_j), A_1(\dot{\Gamma}_j))^{q_0} \}^{\frac{1}{q_1}}.
\end{array}$$

(iv) If $0<q_0 \neq q_1 $ and $q_0 q_1=\infty$, then the conditional K functional
$G(s,\{f_{j,\gamma}\}_{\gamma\in \dot{\Gamma}_{j}}, A_0, A_1)$ is defined in the Equation \eqref{eq:con.fun},
$\tilde{G}(t,f)=t^{-1} G(t,f)$
and $H(t,f)$ is defined in the Equation \eqref{eq:H}.
Further Theorem \ref{th:4} tells us that
\begin{equation*}
\begin{array}{l}K_{\infty}(t,f, \dot{l}^{s_0,q_0} (A_0), \dot{l}^{s_1, \infty} (A_1))= H(t,f) \tilde{G}(H(t,f), f).
\end{array}
\end{equation*}
\begin{equation*}
\begin{array}{l}K_{\infty}(t,f, \dot{l}^{s_0,\infty} (A_0), \dot{l}^{s_1, q_1} (A_1))
= t K_{\infty}(t^{-1},f, \dot{l}^{s_1, q_1} (A_1), \dot{l}^{s_0,\infty} (A_0)).
\end{array}
\end{equation*}

\end{theorem}

Our main Theorem has completely solved Peetre's longstanding Open problem \ref{con:1} on hierarchical psedo-Banach spaces.
%For a long time,
%the study of K functionals has essentially returned to the monolayer structure in some sense.
%The lack of understanding of the structure of the K functional shows
%the lack of understanding of the nonlinear structure of the functions,
%which hinders the further promotion of Marcinkiwicz interpolation theory,
%and also hinders the further application of real interpolation theory.
Theorem \ref{th:main} easily leads to the following three applications.

\begin{remark}
We may take $A_0$ and $A_1$ as Lebesgue spaces $L^{p_0}$ and $L^{p_1}$ or
Lorentz spaces $L^{p_0,\tau_0}$ and $L^{p_1,\tau_1}$.
$(l^{s_0,q_0}(A_0), l^{s_1,q_1}(A_1))$
become Besov spaces $(\dot{B}^{s_0,q_0}_{p_0}, \dot{B}^{s_1,q_1}_{p_1})$
or Besov-Lorentz spaces $(\dot{B}^{s_0,q_0}_{p_0,\tau_0}, \dot{B}^{s_1,q_1}_{p_1,\tau_1})$.
Hence our Theorem here not only solved another open problem posed by Peetre
and but also extend it to much more complex situation.
See Peetre's book \cite{Peetre67book} in 1967 or
page 110 of his English translation \cite{Peetre} in 1976.
\end{remark}

\begin{remark}
More than 70 years after Lorentz's work \cite{Lorentz},
it is fortunate that we can systematically generalize it
by wavelet grid topology.
The interpolation spaces $(l^{s_0,q_0}(A_0), l^{s_1,q_1}(A_1))_{\theta,r}$ in this article
systematically extend Lorentz spaces
defined by rearrangement function in Lorentz \cite{Lorentz}.
Hunt \cite{Hunt1, Hunt} studied the K functional for Lorentz spaces
and Marcinkiewiez's theorem is systematically extended Lorentz spaces.
Our Theorem extend it to which for general Besov hierarchical spaces $l^{s,q}(A)$
and we can consider more general Marcinkiewiez's theorem.
\end{remark}

\begin{remark}
By our main Theorem \ref{th:main}, the inclusion relationship between Besov space and Lorentz space can be conveniently considered.
As a direct application, one can generalize the real interpolation inequalities of Bahouri-Cohen \cite{BaC} and  Chamorro-Lemari\'e \cite{CL}
to more general indices.
Ou-Wang \cite{OW} considered Stein's conjecture.
One can try to extend the estimation of $Ef$ in Ou-Wang's theorem to Lorentz space by real interpolation idea.
\end{remark}

\subsection{Backgroud and motivations
}
%\hskip\parindent
The first reason we're focusing on K functionals right now is that
K functionals come in the form of cost functions in artificial intelligence \cite{HMY, HWJJ}.
The cost function involves the process of minimising two quantities,
hence randomness occurs.
As the spatial structure that defines costs becomes more complex,
so does the structure of randomness.
Hence, it is difficult to separate data into two categories that meet the requirements.
%It is therefore very difficult to classify the data into two categories.
In the case of big data,
the amount of computation required to deal with such extreme value problems is enormous.
In artificial intelligence,
the method of gradient descent and distributed operation are used to deal with such extreme value problems.
We consider the equivalent expression of the K functional.
The explicit expression of equivalent K functional in this paper is obtained by dividing the wavelet coefficients into two parts.
This helps in the classification of data, and thus has value in artificial intelligence.

Our second motivation comes from the fact that the study of interpolation space can improve the study of continuity of operators.
It is difficult to obtain additional log-magnitude estimates.
Because of improving log-magnitude regularity,
Bony's para-product and Lions' compensation compact theory are paid more attention.
Lacey-Thiele \cite{LacTh} solved Calderon's conjecture
by obtaining some degree of log-magnitude regularity using the model sum method.

We take Fourier transform $\mathcal{F}$ as an example to see how complex and real interpolation can improve operator continuity.
Let $L^{p}$ denote a Lebesgue space and let $L^{p'}$ its conjugate space and let $L^{p,q}$ be Lorentz space.
We know $$ \mathcal{F}: L^2 \rightarrow L^2 \mbox{\, and \,} \mathcal{F}: L^1 \rightarrow C_0 .$$
For $1<p<2$, we have to use interpolation to figure out what kind of space $ \mathcal{F}(L^p)$ is in.
In fact, by using complex interpolation, Hausdorff-Young theorem says
\begin{equation} \label{eq:cF} \mathcal{F}: L^p \rightarrow L^{p'}. \end{equation}
Afterwards, Stein \cite{Stein} proposed to consider the cone case of \eqref{eq:cF},
recently, Ou-Wang \cite{OW} solved such open problem.
One can extend such results to Lorentz spaces.
Its another refinement of \eqref{eq:cF} usually associated with Paley.
By applying real interpolation, we have
\begin{equation} \label{eq:rF} \mathcal{F}: L^p \rightarrow L^{p',p}. \end{equation}
For $1<p<2$, $L^{p',p}\subsetneq L^{p'}$.
Lorentz space $L^{p',p}$ is some log level stronger than Lebesgue space $L^{p',p'}=L^{p'}$.
{\bf Compared with complex interpolation, real interpolation improves the regularity of some log order.}
The reason why real interpolation space can improve operator continuity is that
K functional is defined by the lower bound of norm,
which reflects the minimization of energy.
See Peetre \cite{Peetre67book} in 1967.

The third motivation for our interest in interpolation spaces is that
many phenomena related to interpolation have received widespread attention,
but the specific expressions of many general interpolation spaces are not clear.
As early as 1911, Schur \cite{Schur} considered the interpolation properties of operators.
In 1926, M. Riesz \cite{R} proved the first version of the Riesz-Thorin theorem.
In 1948, Salem-Zygmund \cite{SZ} and Thorin  \cite{Thorin}
appeared the study of interpolation on $H^{p}$ spaces.
In 1950, Lorentz \cite{Lorentz} introduced Lorentz spaces.
In 1964, Hunt \cite{Hunt} revealed the relation between the real interpolation space of Lebesgue space and Lorentz space.
The study of interpolation space greatly enriches function space and operator theory and
has important applications in differential equations.
The study of interpolation space has greatly aroused the attention of mathematicians.
J. Bergh-J. L\"ofstr\"om wrote on page 170 of their book \cite{BL} that
Notices A.M.S introduces 12 open problems of interpolation introduced by 12 mathematicians
under the heading ``Problems in interpolation of operations and applications I-II" in
Notices Amer. Math. Soc. See \cite{PCFGLM, SHSTBR}.
These 12 mathematicians are {\bf in order},
J. Peetre, W. Connett, J. Fourier, J. E. Gilbert, J. L. Lions, Benjamin Muckenhoupt,
E. M. Stein, Richard A. Hunt, Robert C. Sharpley, A. Torchinsky, Colin Bennett and N. M. Riviere.
These problems are aimed at {\bf the special phenomena in the process of interpolation}.
The interpolation results have received much attention
and these problems have received the attention they deserve.
See \cite{AK, AKMNP5,BL, BC, BCT,  CD, DP,DY, HS, Holmstedt, HP, LoP,
Peetre67, PS, SaZ, ST, SW, Thorin,  Triebel,TriebelB,YDC,YCP}.

%%%%%%%%%%%%%%%%%%%%%%%%%
%%%%%%%%%%%%%%%%%%%%
%%%%%%%%%%%%%%%%%%%%%%%%%%

\subsection{Main skills and structure of this paper}

In 1974, Cwikel proved that
the Lions-Peetre formula for $(l^{q_0}(A_0), l^{q_1}(A_1))_{\theta,r}$
have no reasonable generalization for any $r\neq q$.
To get the most out of this article,
let's start with the main new techniques developed in this article.
\begin{remark}
We can solve Peetre's longstanding open problem on Besov hierarchical spaces
and obtain Main Theorem \ref{th:main} because
we have adopted different matching methods for different indicators
and the following three new techniques：
(1) {\bf Wavelets have both the basis property and the fully discrete grid property}.
The basis properties allow us to change nonlinearities of the functional,
and the fully discrete mesh allows us to classify grid and consider the grid topology.
Littlewood-Paley decomposition does not have such properties.
(2) {\bf Analysis of nonlinear functional structures based on wavelet coefficients}.
Devore-Popov have considered the relation of K functional
between Besov spaces and the corresponding discrete spaces for frequency in Theorem 6.1 of \cite{DP}.
We consider the more general Besov hierarchical spaces
and use wavelet to generalize the Devore-Popov's result to the wavelet functional case.
Then, by analyzing the wavelet coefficient characteristics of the approximate K-functional,
the wavelet functional is limited to the cuboid functional,
and the cuboid functional is obtained.
Further, based on the quasi-trigonometric inequality for the norm of two function spaces,
a new equivalent functional is obtained by functional calculus: vertex K functional.
According to vertex functional,
the problem of functional is transformed into coefficient data classification problem,
which provides a theoretical foundation for the study of real interpolation structure through lattice topology.
(3) {\bf Analysis of nonlinear lattice topology based on wavelet lattice structure}.
We studied it in four cases based on vertex functional.
Firstly, for $A_0=A_1=A$,
we  take the $A-$norm on the layer grid $\dot{\Gamma}_j$
and then the interpolating norm on the main grid $\mathbb{Z}$.
Secondly, if $A_0\neq A_1$ and $0<q_0=q_1\leq \infty$,
then we use both the structure of the main grid $\mathbb{Z}$
and the quantities $X_j$ and $Y_j$ related to the K functional on the layer grid $\dot{\Gamma}_j$
to give the equivalent K-functional on the full grid $\dot{\Lambda}$.
Thirdly,  if $A_0\neq A_1$ and $0<q_0\neq q_1<\infty$,
then we use power spaces to change the topology structure
and get four intermediate spaces on the full grid $\dot{\Lambda}$ and the layer grid $\dot{\Gamma}_j$ respectively.
The corresponding K functional on $\dot{\Lambda}$ is obtained by using
the commutativity of the infinitesimal functional about the intermediate space and the summation of the main grid $\mathbb{Z}$.
Fourthly and lastly, we consider the cases where $A_0\neq A_1$, $0<q_0\neq q_1$ and $q_0 q_1=\infty$.
We introduce some condition functional under the condition
defined by the nonlinear quantity of the condition functional about the full grid.
The real interpolation space we get is the space with nonlinear topological structure.

The real interpolation space of Besov hierarchical spaces is given by a series of new ideas:
(1) wavelet basis and wavelet grid,
(2) nonlinear functional structure and vertex K functional,
(3) power spaces and nonlinear topology structure  and
(4) condition functional and nonlinearity.

\end{remark}

%\vspace{-0.296cm}
Then we'll look at the structure of this article.
The rest of this article is structured as follows:
In Section \ref{SEC2},
we introduce some preparation knowledge on function spaces and wavelets.
In Section \ref{sec:2}, we present first some preliminaries on K functional and Lorentz spaces.
Then we present some basic knowledge on real interpolation spaces of Besov type spaces and
Peetre's longstanding open problem.

In Section \ref{SEC5},
we consider nonlinear transformations of functional.
In Subsection \ref{sec:4.1}, we discuss
wavelet functional and cuboid functional.
Firstly, we discretize K functional using wavelet and get wavelet functional.
Then, by analyzing the wavelet coefficient characteristics of the approximation wavelet functional,
it is obtained that cuboid functional and we get wavelet basis is unconditional basis of interpolation space.
This is the extension and refinement of the main Theorem 6.1 of Devore-Popov in \cite{DP}.
Based on the characteristics that both spatial norms satisfy the quasi-triangle inequality,
in Subsection \ref{SEC6}, through functional calculus,
we obtain vertex K functional.
The calculation of K functional is converted into classification of data generated by wavelet coefficients.
This is the theoretical foundation for considering the lattice topological structure of K functional.

\vspace{0.1cm}
For the remaining three sections,
we consider the topology of real interpolation in four cases
and give concrete expressions of K-functional.
In Section \ref{SEC7},
we consider the case $A_0=A_1$ or $q_0=q_1$.
In Section \ref{SEC:5.1},
we consider the case $A_0=A_1=A$
and take $A$-norm on the layer grid and get sequence $(F_j)_{j\in \mathbb{Z}}$.
For $(F_j)_{j\in \mathbb{Z}}$, we take interpolation norm on the main grid
and get the specific expression of K-functional.
In Section \ref{SEC:5.2}, we obtain $X_j(t)$ and $Y_j(t)$ on a layer grid by using vertex K functional,
and then consider their power relation,
and then give the topology structure of K functional in the case
$0<q_0=q_1\leq \infty$.
In Section \ref{sec:6x}, we calculate the K functional for $A_0\neq A_1$ and $0<q_0\neq q_1<\infty$.
In Subsection \ref{subsec:6.1}, we introduce the main techniques in analysing the topology structure.
We establish the relationship of topological structure of power spaces and the topological compatibility of intermediate spaces.
In Subsection \ref{subsec:6.2},
the corresponding K functional topology is obtained in three steps based on vertex K functional.
We first use the power relation to give the K functional of $(Y_0^j, Y_1^j)$,
and then use the commutativity of the functional and the summation
to get the K functional of $(X_0,X_1)$.
Finally, the K functional of $(\dot{l}^{s_0,q_0}(A_0), \dot{l}^{s_1,q_1}(A_1))$
is obtained by using the power relation again.
In Section \ref{sec:77}, we calculate the K functional for $A_0\neq A_1$ and $q_0  q_1= \infty$.
We use conditional G functionals and the nonlinearity of threshold of conditional functional
to give concrete expressions for K functionals.

\section{Function spaces and wavelets}\label{SEC2}

In this section, we introduce all function spaces relevant to this paper and their wavelet characterizations.

\subsection{Distributions and Littlewood-Paley decomposition}
After the development of the 1970s, most function spaces form the theory of systems
by using Littlewood-Paley decomposition.
Triebel \cite{TriebelB} systematically classifies them as Besov spaces and Triebel-Lizorkin spaces.
Morrey spaces are introduced in 1938.
In this century, Carleson measure and Hausdorff capacity are systematically applied to the study of function spaces.
And systematically form Besov Morrey spaces, Triebel-Lizorkin-Morrey spaces and
their predual spaces Besov-Hausdorff spaces and Triebel-Lizorkin-Hausdorff spaces.
See \cite{SW, TriebelB, Yang1,YCP} and \cite{YSY}.
In this paper, we only find out the real interpolation spaces of Besov hierarchical spaces
without considering Carleson measure and Hausdorff capacity.
Due to the existence of minimal energy processes,
the concrete form of most real interpolation spaces is very complicated.

Let $S$ be the space of all Schwartz functions on $\mathbb{R}^{n}$
and let $S_{0}$ be the space of all Schwartz functions $f$ on $\mathbb{R}^{n}$
such that $\int x^{\alpha} f(x) dx=0, \forall \alpha\in \mathbb{N}^{n}$.
Homogeneous and non-homogeneous Besov type spaces are based on distribution theory.
The space of all tempered distributions on $S(\mathbb{R}^n)$
which is equipped with the weak-$\ast$ topology is denoted by $\mathcal{S}'(\mathbb{R}^n)$.
Respectively, the space of all tempered distributions on $S_{0}(\mathbb{R}^{n})$
which is equipped with the weak-$\ast$ topology is denoted by
$\mathcal{S}_{0}'(\mathbb{R}^n)$.
Hence $\mathcal{S}'(\mathbb{R}^{n})\subset \mathcal{S}_{0}'(\mathbb{R}^n).$
Denote the space of all polynomials on $\mathbb{R}^{n}$ by $P(\mathbb{R}^{n})$,
sometimes one denotes $\mathcal{S}_{0}'(\mathbb{R}^n)=\mathcal{S}'(\mathbb{R}^{n})\backslash P(\mathbb{R}^{n})$.
Given a nonnegative function $\widehat{\varphi}(\xi)\in S(\mathbb{R}^{n})$ such that
${\rm supp}\;\widehat{\varphi}\subset\{\xi\in \mathbb{R}^{n}:|\xi|\leq 2\}$ and $\hat{\varphi}(\xi)=1$ if $|\xi|\leq\frac{1}{2}$.
Define
$$\varphi_{u}(x)=2^{n(u+1)}\varphi(2^{u+1}x)-2^{nu}\varphi(2^{u}x).$$
Then the family of functions $\hat{\varphi}(\xi), \{\hat{\varphi}_{u}(\xi)\}_{u\in \mathbb{Z}}$ satisfy that

\begin{equation*}
\left\{ \begin{aligned}
&\hat{\varphi}(\xi)+\sum\limits_{u\geq 0}\hat{\varphi}_{u}(\xi)=1,\;{\rm for}\;{\rm all} \;\xi\in \mathbb{R}^{n}.&\\
&\sum\limits_{u=-\infty}^{+\infty}\hat{\varphi}_{u}(\xi)=1,\;{\rm for}\;{\rm all} \;\xi\in \mathbb{R}^{n} { \rm \; and \;} \xi\neq 0.&
\end{aligned}
\right.
\end{equation*}
Further,
\begin{equation*}
\left\{ \begin{aligned}
&{\rm supp}\;\hat{\varphi}_{u}\subset\{\xi\in \mathbb{R}^{n},\;\frac{1}{2}\leq 2^{-u}|\xi|\leq2\}.& \\
&|\hat{\varphi}_{u}(\xi)|\geq C>0, \;{\rm if} \;\frac{1}{2}<C_{1}\leq 2^{-u}|\xi|\leq C_{2}<2.&\\
&|\partial^{k}\hat{\varphi}_{u}(\xi)|\leq C_{k}2^{-u|k|}, \;{\rm for}\; {\rm any}\;k\in \mathbb{N}^{n}.&
\end{aligned}
\right.
\end{equation*}

If $f\in S_{0}'(\mathbb{R}^{n})$, then $\forall u\in \mathbb{Z}$, define $ f^{u}=\varphi_{u}\ast f$.
If $f\in S'(\mathbb{R}^{n})$, then $\forall u\geq 1$, define $ f_{u}=\varphi_{u}\ast f$
and define $f_{0}= (\varphi+ \varphi_0)\ast f. $
Where $f^{u}, f_{u}$ are called the $u-$th dyadic block of the Littlewood-Paley decomposition of $f$.
%Without causing confusion, we mix $f^{u}$ and $f_{u}$.

For $s\in \mathbb{R}, 0<p\leq \infty$,
denote $f$ belongs to the Sobolev spaces $\dot{W}^{s,p}$,
if $(-\Delta)^{-\frac{s}{2}}f \in L^{p}$.
Besov spaces are defined by $l^{s,q}(L^{p})$, the weighted $l^{s,q}$ norm of the sequence of Lebesgue norms $L^{p}$ over a ring,
or equivalents, by $l^{q}(\dot{W}^{s,p})$, the sequence $l^{q}$ norm of the Sobolev norm $\dot{W}^{s,p}$ over a ring.
A more complex space than Besov space is to replace $L^p$ with a more general  $A$.
We recall here the definition of Besov type spaces $\dot{l}^{s,q}(A)$ and $l^{s,q}(A)$.
See \cite{CDL, Peetre, Triebel, TriebelB, Yang1}.
That is to say,

\begin{definition}\label{de1.1}
Let  $0< p, q\leq\infty$ and $s\in \mathbb{R}$. We have\\
{\rm(\romannumeral1)} $f(x)\in B^{s,q}_{p}=l^{s,q}(L^p)$,  if $f\in S'(\mathbb{R}^{n})$ and  $\left[\sum\limits_{u\geq 0}2^{us q}\|f_{u}(x)\|_{L^{p}}^{q}\right]^{\frac{1}{q}}<\infty$.\\
{\rm(\romannumeral2)}
$f(x)\in \dot{B}^{s,q}_{p}=\dot{l}^{s,q}(L^p)$,   if $f\in S_{0}'(\mathbb{R}^{n})$ and   $\left[\sum\limits_{u\in \mathbb{Z}}2^{us q}\|f^{u}(x)\|_{L^{p}}^{q}\right]^{\frac{1}{q}}<\infty$.\\
{\rm(\romannumeral3)}
$f(x)\in l^{s,q}(A)$,   if $f\in S'(\mathbb{R}^{n})$ and   $\left[\sum\limits_{u\geq 0}2^{us q}\|f^{u}(x)\|_{A}^{q}\right]^{\frac{1}{q}}<\infty$.\\
{\rm(\romannumeral4)}
$f(x)\in \dot{l}^{s,q}(A)$,   if $f\in S_{0}'(\mathbb{R}^{n})$ and   $\left[\sum\limits_{u\in \mathbb{Z}}2^{us q}\|f^{u}(x)\|_{A}^{q}\right]^{\frac{1}{q}}<\infty$.\\
When $q=\infty$,  it should be replaced by the supremum.
\end{definition}

For $0<p<\infty$ and $0<q\leq \infty$, we have
$$\|f^u\|_{L^{p}}\sim \|f^{u}\|_{\dot{F}^{0,q}_{p}} \sim \|f^{u}\|_{\dot{B}^{0,q}_{p}} \mbox{\, and \,}
\|f^{u}\|_{L^{\infty}} \sim \|f^{u}\|_{\dot{B}^{0,q}_{\infty}}.$$
We recall the definition of  Besov-Lorentz spaces $\dot{B}^{s,q}_{p,r}= \dot{l}^{s,q}(L^{p,r})$
which have been studied in \cite{BL, Peetre, YCP}.
Let $E\subset \mathbb{R}^{n}$, we denote $|E|$ the Lebesgue measure of $E$.
\begin{definition}\label{de1.2}
Assume that $0< p, q,r\leq\infty$, $s\in \mathbb{R}$ and $u,v\in \mathbb{Z}$. Then\\
(\romannumeral1)  $f(x)\in \dot{B}^{s,q}_{p,r}$, if $f\in S_{0}'(\mathbb{R}^{n})$ and $$\left[\sum\limits_{u\in \mathbb{Z}}2^{uqs}\left(\sum\limits_{v\in \mathbb{Z}}2^{rv}|\{x\in \mathbb{R}^{n}:|f^{u}(x)|>2^{v}\}|^{\frac{r}{p}}\right)^{\frac{q}{r}}\right]^{\frac{1}{q}}<\infty.$$
(\romannumeral2)  $f(x)\in B^{s,q}_{p,r}$, if $f\in S'(\mathbb{R}^{n})$ and  $$\left[\sum\limits_{u\geq 0}2^{uqs}\left(\sum\limits_{v\in \mathbb{Z}}2^{rv}|\{x\in \mathbb{R}^{n}:|f_{u}(x)|>2^{v}\}|^{\frac{r}{p}}\right)^{\frac{q}{r}}\right]^{\frac{1}{q}}<\infty.$$
As $p=\infty$ or $q=\infty$ or $r=\infty$,  it should be modified by  supremum.
\end{definition}

The definition of the above spaces are independent of the choice of  $\varphi_{u}$,
see \cite{BL,TriebelB,Yang1,Yang2,YCP}.

\subsection{Wavelets and Besov hierarchical spaces}

The norm of the traditional function space reflects a linear structure.
Littlewood-Paley decomposition and wavelet decomposition can almost replace each other when only such a structure is considered.
{\bf But the interpolation of the function,
which uses the lowest energy, has a nonlinear structure.}
Littlewood-Paley decomposition $f\rightarrow \{f*\phi_u\}$  considers only the discretization of frequencies,
not decomposition for position and oscillation.
Our proof of real interpolation must rely on the vertex K functional.
If only Littlewood-Paley decomposition is used, we cannot get vertex K functionals.

We use regular orthogonal tensorial wavelets. See \cite{Meyer, Yang1, YCP}.
For dimension 1, let $\Phi^{0}$ be the father wavelet and let $\Phi^{1}$ be the mother wavelet.
For dimension $n$, let directional set $\Xi= \{0,1\}^{n}$. For $ \epsilon\in \Xi$, denote $\Phi^{\epsilon}(x) = \prod\limits^{n}_{i=1} \Phi^{\epsilon_{i}}(x_{i})$.
For $j\in \mathbb{Z}, \epsilon\in \Xi, k\in \mathbb{Z}^{n}$, let
$\Phi^{\epsilon}_{j,k}(x) = 2^{\frac{nj}{2}} \Phi^{\epsilon}( 2^{j} x-k).$
$\epsilon$ denotes the property of whether the integral of the direction function is zero for each coordinate axis,
$j$ is a quantity related to frequency, $k$ is a quantity related to position.
To characterize both homogeneous spaces and non-homogeneous spaces, we need the following grid notations:
$$\dot{\Xi}= \{ \epsilon: \epsilon\in \{0,1\}^{n} \backslash \{(0,\cdots,0)\}\},
\Xi_{0}= \Xi  {\rm \, and \, for \,} j\geq 1, \Xi_{j}=\dot{\Xi}$$
$$\dot{\Gamma}= \{ \gamma= (\epsilon, k), \epsilon\in \dot{\Xi}, k\in \mathbb{Z}^{n}\},
\Gamma = \{ \gamma= (\epsilon, k), \epsilon\in \Xi_{0}, k\in \mathbb{Z}^{n}\}.
$$
$$ \Gamma_{0} = \{(0,\gamma), \gamma\in \Gamma\},
{\rm\, and \, for \,} j\geq 1, \Gamma_{j} = \{(j,\gamma), \gamma\in \dot{\Gamma}\}.$$
$$ \forall j\in \mathbb{Z},  \dot{\Gamma}_{j} = \{(j,\gamma), \gamma\in \dot{\Gamma}\},
\mathbb{N}= \{j\in \mathbb{Z}, j\geq 0\}.$$
$$\dot{\Lambda}= \{(j,\gamma), j\in \mathbb{Z}, \gamma\in \dot{\Gamma}\} {\rm \, and \,}
\Lambda = \{(j,\gamma), j\in \mathbb{N}, \gamma\in \Gamma_{j}\}.$$

We don't distinguish between $(\epsilon,j,k)$ and $(j,\gamma)$ and denote
$\Phi_{j,\gamma}(x) = \Phi^{\epsilon}_{j,k}(x)$.
For function $f$, denote $$f_{j,\gamma}= \langle f, \Phi_{j,\gamma}\rangle .$$
When we consider homogeneous spaces, we use $\dot{\Lambda}$.
When we consider non-homogeneous spaces, we use $\Lambda$.

We specifically point out that $j$ represents the frequency layer lattice of the object being studied,
and we call $\mathbb{Z}$ the main grid.
We denote $\Gamma_{j}$ the layer grid which represents the direction and position lattice.
We call $\dot{\Lambda}$ the full grid.
In this paper, we consider not only the wavelet coefficients corresponding to the functional,
but also the topological structures of the three meshes $\mathbb{Z}, \dot{\Gamma}_{j}$ and $\dot{\Lambda}$.
By using wavelets knowledge in \cite{Meyer, Yang1},
the Theorem \ref{lem:dtod} below provides the basis for studying the discretization of the general distribution spaces.
\begin{theorem}\label{lem:dtod}
If we use Meyer wavelets, then wavelet theory shows \\
(i) $f\rightarrow \{f_{j,\gamma}\}_{(j,\gamma)\in \Lambda}$ is an isomorphic mapping from $S'(\mathbb{R}^{n})$ to $\mathbb{C}^{\Lambda}$.
\\
(ii) $f\rightarrow \{f_{j,\gamma}\}_{(j,\gamma)\in \dot{\Lambda}}$ is an isomorphic mapping from $S'_{0}(\mathbb{R}^{n})$ to $\mathbb{C}^{\dot{\Lambda}}$.
\end{theorem}
So without causing confusion, we do not distinguish between the function $f$ and the sequence $\{f_{j,\gamma}\}$.
Wavelet functions $\Phi_{j,\gamma}$ are good functions.
For $f= \sum\limits_{j,\gamma} f_{j,\gamma} \Phi_{j,\gamma}$,
only the variation of the wavelet coefficients $f_{j,\gamma}$ affects the norm of the distribution $f$.
The above lemma transforms the nonquantifiable distribution $f\in S'(\mathbb{R}^{n})$ or $f\in S'_{0}(\mathbb{R}^{n})$
into countable dimensional wavelet coefficients  $\{f_{j,\gamma}\}_{(j,\gamma)\in \Lambda}\in C^{\Lambda}$  or
$\{f_{j,\gamma}\}_{(j,\gamma)\in \dot{\Lambda}}\in \mathbb{C}^{\dot{\Lambda}}$.
It allows us to introduce the definition of a function with wavelet coefficients on a discrete set.
\begin{definition}
We say ${\rm Supp}_{w} f \subset E\subset \dot{\Lambda}$,
if $f(x)= \sum\limits_{(j,\gamma)\in E} f_{j,\gamma} \Phi_{j,\gamma}(x)$.
\end{definition}
It provides a theoretical basis for us to give the vertex K functional based on the wavelet coefficients.
Let $\chi(x)$ denote the characteristic function on the unit cube $[0,1]^{n}$.
For $j\in \mathbb{Z}$, denote $f^{j}(x)= 2^{\frac{nj}{2}} \sum\limits_{\gamma\in \dot{\Gamma} } |f_{j,\gamma}| \chi (2^{j}x-k).$
For $j\in \mathbb{N}$, denote $f_{j}(x)= 2^{\frac{nj}{2}} \sum\limits_{\gamma\in \Gamma_{j} } |f_{j,\gamma}| \chi (2^{j}x-k)$
and denote $$f^{\infty}_{j}(x)= 2^{\frac{nj}{2}} \sup\limits_{\epsilon\in \Xi_{j}} \sum\limits_{k\in \mathbb{Z}^{n} } |f_{j,\gamma}| \chi (2^{j}x-k).$$
We recall the wavelet characterization of Besov-Lorentz spaces where $A=L^{p,r}$.
According to \cite{YCP}, we have the following characterization for Besov-Lorentz spaces:
\begin{lemma}\label{lem:Besov-Lorentz}
Assume that $0< p, q,r\leq\infty$, $s\in \mathbb{R}$.

(i) $f(x)= \sum\limits_{(j,\gamma)\in \dot{\Lambda}} f_{j,\gamma} \Phi_{j,\gamma}\in \dot{B}^{s,q}_{p,r} \Leftrightarrow
\sum\limits_{j\in \mathbb{Z}} 2^{jsq} \{ \sum\limits_{u\in \mathbb{Z}} 2^{ur}
|\{x: f^{j}(x)> 2^{u}\}|^{\frac{r}{p}}\}^{\frac{q}{r}} < +\infty.$
Particularly, if $p=r$, then $\dot{B}^{s,q}_{p,r}= \dot{B}^{s,q}_{p}$ and
$$f(x)= \sum\limits_{(j,\gamma)\in \dot{\Lambda}} f_{j,\gamma} \Phi_{j,\gamma}\in \dot{B}^{s,q}_{p} \Leftrightarrow
\sum\limits_{j\in \mathbb{Z}} 2^{jq(s+\frac{n}{2}-\frac{n}{p})} (\sum\limits_{\gamma\in \dot{\Gamma}}
|f_{j,\gamma}|^{p})^{\frac{q}{p}} < +\infty.$$

(ii) $f(x)= \sum\limits_{(j,\gamma)\in \Lambda} f_{j,\gamma} \Phi_{j,\gamma}\in B^{s,q}_{p,r} \Leftrightarrow
\sum\limits_{j\in \mathbb{N}} 2^{jsq} \{ \sum\limits_{u\in \mathbb{Z}} 2^{ur}
|\{x: f_{j}(x)> 2^{u}\}|^{\frac{r}{p}}\}^{\frac{q}{r}} < +\infty. $
Particularly, if $p=r$, then $B^{s,q}_{p,r}= B^{s,q}_{p}$ and
$$f(x)= \sum\limits_{(j,\gamma)\in \Lambda} f_{j,\gamma} \Phi_{j,\gamma}\in B^{s,q}_{p} \Leftrightarrow
\sum\limits_{j\geq 0} 2^{jq(s+\frac{n}{2}-\frac{n}{p})} (\sum\limits_{\gamma\in \Gamma_{j}}
|f_{j,\gamma}|^{p})^{\frac{q}{p}} < +\infty.$$

\end{lemma}

A psedo-Banach space or even Banach space may has no unconditional basis.
For example, $L^{p}$ has no unconditional basis when $0<p\leq 1$ or $p=\infty$.
\begin{remark}
For pseudo-Banach space $A$, $\forall j\in \mathbb{Z}$,
denote $A^{j}=A(\dot{\Gamma}_{j}) = \{f(x)\in A, f(x)=\sum\limits_{\gamma\in \dot{\Gamma}_{j}} f_{j,\gamma}\Phi_{j,\gamma}(x)\}$.
May be $\{\Phi_{j,\gamma}(x)\}_{(j,\gamma)\in \dot{\Lambda} }$ is not an unconditional basis for pseudo-Banach space $A$.
But when we consider $A(\dot{\Gamma}_{j})$, the situation will change.

(i) We know that wavelet basis is not an unconditional basis for $L^{\infty}$ and $L^{p}(0<p\leq 1)$.
But for $0<p<\infty, 0<q\leq \infty$ and $j\in \mathbb{Z}$, we know
$$\|\{f_{j,\gamma}\}_{\gamma \in \dot{\Gamma} }\|_{L^{p}}\sim \|\{f_{j,\gamma}\}_{\gamma \in \dot{\Gamma} }\|_{\dot{F}^{0,q}_{p}}
\sim \|\{f_{j,\gamma}\}_{\gamma \in \dot{\Gamma} }\|_{\dot{B}^{0,q}_{p}},$$
$$\|\{f_{j,\gamma}\}_{\gamma \in \dot{\Gamma} } \|_{L^{\infty}} \sim \|\{f_{j,\gamma}\}_{\gamma \in \dot{\Gamma} }\|_{\dot{B}^{0,q}_{\infty}}.$$
Hence on the ring, the norm is determined by the absolute value of their wavelet coefficients.

(ii) The condition that $\{\Phi_{j,\gamma}(x)\}_{\gamma\in \dot{\Gamma}_{j}}$ is a group of unconditional basis in $A(\dot{\Gamma}_{j})$
is equivalent to the following condition:
\begin{equation}\label{eq:absolute}
\forall j\in \mathbb{Z},
\|\{f_{j,\gamma}\}_{\gamma \in \dot{\Gamma} } \|_{A} \sim \|\{|f_{j,\gamma}|\}_{\gamma \in \dot{\Gamma} }\|_{A}.
\end{equation}
The above condition \eqref{eq:absolute} is implicit in Peetre's book \cite{Peetre},
and it's weaker than the condition that A has an unconditional basis.
We did not attempt to prove that any pseudo-Banach space constraint on a ring automatically satisfies the condition \eqref{eq:absolute},
in order to reduce the length of the paper and
make readers pay more attention to the technique of giving K functional expressions in the general cases.

\end{remark}

For Besov hierarchical spaces, we have the following wavelet characterization:
\begin{theorem} \label{lem:cbesov}
Assume that $0< q \leq\infty$, $s\in \mathbb{R}$.
If $A$ satisfies \eqref{eq:absolute}, then
$$f(x)= \sum\limits_{(j,\gamma)\in \dot{\Lambda}} f_{j,\gamma} \Phi_{j,\gamma}\in \dot{l}^{s,q}(A) \Leftrightarrow
\sum\limits_{j\in \mathbb{Z}} 2^{jqs}
\|\{|f_{j,\gamma}| \}_{\gamma\in \dot{\Gamma}} \|_{A} ^{q} < +\infty.$$
\end{theorem}

\section{Preliminaries on K functional}\label{sec:2}
Peetre \cite{Peetre} introduced K functional to considered real interpolation spaces.
In modern learning algorithms, K functional appears as a cost function with regularization term.
See \cite{LCWG}.
In this section, we present some preliminaries on K functional.

\subsection{K functional}\label{sec:2.1}

Peetre \cite{Peetre} introduced K functional $K(t,f)$ to study real interpolation spaces.
We recall that if $(A_{0},\;A_{1})$ is a pair of quasi-normed spaces which are continuously embedded in a Hausdorff space $X$.
Suppose that $A_0$ and $A_1$ are compatible normed vector spaces.
For any threshold $t>0$,
$\forall f\in A_0 + A_1$ and $f_0\in A_0$,
define  a set of dynamically varying norms
$$ K(t,f, A_0, A_1,f_0) =  \|f_0\|_{A_0} + t \|f-f_0\|_{A_1} $$
for all $f\in A_{0}+A_{1}$ with $f_{0}\in A_{0}$ and $f_{1}\in A_{1}$.
$K$-functional is the lower bound of this dynamically varying norm, i.e.
$$K:= K(t,f, A_0, A_1) = \inf_{f=f_0+ f_1}
\{ \|f_0\|_{A_0} + t \|f_1\|_{A_1}\}.$$
$K(t,f,  A_0, A_1)$ is a nonnegative, increasing concave function of $t$.
K functionals have the following commutative properties:
\begin{lemma} \label{lem:CK}
For $t>0$, we have
$K(t,f, A_0, A_1)= t K(t^{-1}, f, A_1, A_0).$
\end{lemma}

Based on K functional, we can define the corresponding real interpolation space. For $0<\theta<1$,
\begin{equation}\label{eq:theta.infty}
\|f\|_{(A_0,A_1)_{\theta,\infty,K}}= \sup\limits_{t>0} t^{-\theta} K(t, f, A_0, A_1).$$
$$(A_0, A_1)_{\theta,\infty, K} = \{f: f\in A_0 +A_1, \|f\|_{(A_0,A_1)_{\theta,\infty, K}}<\infty\}.
\end{equation}
Further, for $0<\eta<\infty$, define
$$\|f\|_{(A_0,A_1)_{\theta,\eta,K}}= \{\int^{\infty}_{0} [ t^{-\theta} K(t, f, A_0, A_1)] ^{\eta}
\frac{dt}{t} \}^{\frac{1}{\eta}}.$$
\begin{equation}\label{eq:theta.eta}
(A_0, A_1)_{\theta,\eta, K} = \{f: f\in A_0 +A_1, \|f\|_{(A_0,A_1)_{\theta,\eta, K}}<\infty\}.
\end{equation}

The interpolation spaces have the following commutative properties:
\begin{lemma}\label{lem:commutative}
Given $0<\theta<1, 0<\eta\leq \infty.$
\begin{equation*}\label{eq:2221}
\|f\|_{(A_0,A_1)_{\theta,\eta, K}}= \|f\|_{(A_1,A_0)_{1-\theta,\eta, K}}.
\end{equation*}
\end{lemma}

Holmstedt-Peetre \cite{HP} introduced $K_{\xi}$ functional to improve the situation.
Let $0<\xi<\infty$, define the functional $K_{\xi}$ by
$$K_{\xi}=K_{\xi}(t,f,A_{0},A_{1})=\inf\limits_{f=f_{0}+f_{1}}(\|f_{0}\|_{A_{0}}^{\xi}+t^{\xi}\|f_{1}\|_{A_{1}}^{\xi})^{\frac{1}{\xi}},$$
and
$$
\|f\|_{(A_{0},A_{1})_{\theta,q,K_{\xi}}}=\left[\int_{0}^{\infty}[t^{-\theta}K_{\xi}(t,f,A_{0},A_{1})]^{q}\frac{dt}{t}\right]^{\frac{1}{q}}.
$$
Holmstedt-Peetre \cite{HP} tell us that the real interpolation spaces formed by
the two kinds of K functionals are equivalent.
\begin{lemma}\label{le3.1}
{\bf Holmstedt-Peetre Theorem.}
Let $(A_{0},A_{1})$ be a couple of quasi-normed spaces.
For any $0<\xi<\infty$, we have
$$\|f\|_{(A_{0},A_{1})_{\theta,q,K}}\sim\|f\|_{(A_{0},A_{1})_{\theta,q,K_{\xi}}}.$$
\end{lemma}

Many existing interpolation results use the following Lions-Peetre's iterative Theorem
to get complex real interpolation spaces from simple structure real interpolation spaces.
\begin{lemma}\label{lem:LP}{\bf Lions-Peetre's iterative Theorem, 1964.}
For $0<\theta_0, \theta_1, \eta<1, 0<q_0, q_1\leq \infty, 1\leq q\leq \infty, \theta= (1-\eta) \theta_0 + \eta \theta_1$, we have
$$\{ (A_0, A_1) _{\theta_0, q_0}, (A_0, A_1)_{\theta_1, q_1}\}_{\eta, q} = (A_0, A_1)_{\theta, q}.$$
\end{lemma}

For the iteration real interpolation spaces,
Holmstedt \cite{Holmstedt} obtained the following composite rules for K-functionals.
\begin{lemma} \label{th:Holmstedt}
{\bf Holmstedt Theorem, 1970.}
Let $A_0$ and $A_1$ be a couple of quasi-normed spaces and for $i=1,2$, put
$E_{i}= (A_0, A_1)_{\theta_i, q_i} $.
If $0<\theta_0<\theta_1<1$, $\eta= \theta_1-\theta_0$ and $0<q_0,q_1\leq \infty$, then
\begin{equation*}
\begin{aligned}
K(t,f, E_0, E_1)
&=\{ \int^{t^{\frac{1}{\eta}}}_{0} (s^{-\theta_0} K(s, f, A_0, A_1))^{q_0} \frac{ds}{s}\}^{\frac{1}{q_0}}\\
&+ t \{ \int^{\infty}_{t^{\frac{1}{\eta}}} (s^{-\theta_1} K(s, f, A_0, A_1))^{q_1} \frac{ds}{s}\}^{\frac{1}{q_1}}.
\end{aligned}
\end{equation*}
\end{lemma}

\subsection{Lorentz spaces}\label{sec:2.1x}

In the 1950s, Lorentz \cite{Lorentz}
considered distribution functions and
non-increasing rearrangement functions
and introduce Lorentz spaces to extend Lebesgue spaces.
More than a decade later,
Hunt \cite{Hunt1, Hunt} extend Marcinkiewiez theorem to Lorentz spaces and
revealed the relation between the real interpolation space of Lebesgue space and Lorentz space.
And this eventually forms the Lorentz space theory of real interpolation spaces.
The study of real interpolation spaces greatly expands the Lebesgue space.
The real interpolation spaces have particular nonlinear structures based on the minimum energy,
so it can improve the continuity of the operator,
so many people pay attention to it.

Lorentz \cite{Lorentz} has considered the distribution function and the non-increasing rearrangement function of Lebesgue functions
and introduced Lorentz spaces.
For set $E$, let $|E|$ denote the measure of $E$.
Consider the following right continuons non-increasing functions:\\
$$\sigma_{f}(\lambda)= |\{x: |f(x)|>\lambda\}| {\rm\, and \,} f^{*}(\tau)= \inf \{\lambda: \sigma_{f}(\lambda)\leq \tau\}.$$
Lorentz space $L^{p,q}$ is defined by non-increasing function $f^{*}$:
\begin{definition}\label{def:2.7}
Given $0<p,q<\infty$.\\
(i) $f\in L^{p,q}\Leftrightarrow \{ \frac{q}{p} \int^{\infty}_{0}  (t^{\frac{1}{p}} f^{*}(t))^{q} \frac{dt}{t} \}^{\frac{1}{q}}< +\infty.$\\
(ii) $f\in L^{p,\infty} \Leftrightarrow \sup\limits_{t>0} t^{\frac{1}{p}} f^{*}(t)< +\infty.$
\end{definition}
Since Lorentz introduced Lorentz space in \cite{Lorentz},
there has been sporadic work to extend this space,
but no systematic work has been seen.
In this paper, Lorentz space is extended systematically by a special grid topology,
and a unified description of Besov type space and real interpolation space is obtained.
Hunt \cite{Hunt} revealed the relation between Lorentz spaces and the real interpolation spaces of Lebesgue spaces.
By using the rearrangement functions, the K functional of Lebesgue function has following property.
See \cite{BL, Lorentz, Peetre}.
\begin{lemma}\label{lem:2.5}
{\bf Hunt Theorem, 1966.}\\
(i) If $0<p<\infty$, then  \,\,  $K(t,f, L^{p}, L^{\infty}) = [ \int ^{t^{p}}_{0} | f^{*}(\tau)| ^{p} d\tau ]^{\frac{1}{p}}.$\\
(ii) If $0<p_{0}<p_{1}<\infty$ and $\frac{1}{\alpha}= \frac{1}{p_{0}}- \frac{1}{p_{1}}$, then\\
$$K(t,f, L^{p_{0}}, L^{p_{1}}) = [ \int ^{t^{\alpha}}_{0} | f^{*}(\tau)| ^{p_{0}} d\tau ]^{\frac{1}{p_{0}}}
+ t [ \int ^{\infty}_{t^{\alpha }} | f^{*}(\tau)| ^{p_{1}} d\tau ]^{\frac{1}{p_{1}}}.$$
\end{lemma}
According to Lemma \ref{lem:2.5},
the real interpolation spaces of Lebesgue spaces are Lorentz spaces in Definition \ref{def:2.7}.
See \cite{BL, CDL, Triebel, TriebelB}.
\begin{lemma}\label{lem:LL}
(i) If $0<\theta<1, 0<p_{0}\neq p_{1}, q,q_0,q_1 \leq \infty$ and $\frac{1}{p}= \frac{1-\theta}{p_0} + \frac{\theta}{p_1},$ then
$$( L^{p_{0}, q_0}, L^{p_{1}, q_1})_{\theta,q}  = L^{p,q}.$$
(ii) If $0<\theta<1, 0<p, q,q_0,q_1 \leq \infty$ and $\frac{1}{q}= \frac{1-\theta}{q_0} + \frac{\theta}{q_1},$ then
$$( L^{p, q_0}, L^{p, q_1})_{\theta,q}  = L^{p,q}.$$
\end{lemma}

Lions-Peetre's reiteration theorem cannot be used to study the real interpolation
of general function spaces for general indices.
See  Cwikel's work   \cite{C} in 1974.
For example, we did not know the real interpolation spaces for the general indices of the following function spaces:
Besov spaces, Triebel-Lizorkin spaces,
Besov-Lorentz spaces and Tribel-Lizorkin-Lorentz spaces.
As we know, the corresponding interpolation spaces are far from clear.
See \cite{BL, Peetre, Triebel, TriebelB}.
Yang-Cheng-Peng \cite{YCP} have used wavelets, vector valued Fefferman-Stein maximum operator and Whitney decomposition
to consider some properties of Tribel-Lizorkin-Lorentz spaces and Besov-Lorentz spaces.
Similar to Hunt's work \cite{Hunt1},
we can extend Marcinkiewicz's theorem to Besov hierarchical spaces $\dot{l}^{s,q}(A)$.

\subsection{Real interpolation spaces for Besov spaces}

Besov spaces, developed in the 1950s, are very important
in harmonic analysis, operator theory, approximation theory and differential equations.
For decades, there have been many works focusing on Besov spaces and interpolation spaces and their applications,
see \cite{AK,AKMNP5, BaC, BL, BC, BCT, CL, CD,DP,DY, HS, HP, LiuYY, ST, Triebel,TriebelB,YDC,YCP, ZYY}.
We know that the Besov space $B^{s,q}_{p}= l^{s,q}(L^{p})$.
A more general analogy to Besov spaces,
is to replace $L^p$ with a quasi Banach space $A$ and we consider $l^{s,q}(A)$.
One of the cases is to take $A$ as some Lorentz space $L^{p,\tau}$
and $B^{s,q}_{p,\tau}= l^{s,q}(L^{p,\tau})$ is Besov-Lorentz spaces
which has been studied in \cite{Peetre} and \cite{YCP}.
Triebel-Lizorkin-Lorentz spaces and Besov-Lorentz spaces
have been applied to the study of Navier-Stokes equations in
\cite{BHT} and \cite{HS}.
Lorentz space reflects the lowest energy characteristics and are also used in the study of equations.

For decades, there has been many work focusing on Besov type spaces and interpolation spaces and their applications,
see \cite{AK,AKMNP5, BL, BC, BCT, CD,DP,DY, HS, HP, ST, Triebel,TriebelB,YDC,YCP}.
When $p_0=p_1=p$ and $r=q$, $(\dot{B}_{p}^{s_{0},q_{0}}, \dot{B}_{p}^{s_{_{1}},q_{1}})_{\theta, r}$ are still Besov spaces.
D.C. Yang has considered such cases for metric spaces, see \cite{BL,TriebelB,YDC}.
Q.X. Yang has been paying attention to real interpolation spaces since 2000,
and several European mathematicians have inquired about Yang-Cheng-Peng's results  in various ways.
See \cite{BC, BCT, CD, HS, YCP}.
However, due to the existence of multiple-nonlinear winding structure in real interpolation space,
even for the real interpolation space of Besov space, a very few index spaces have the concrete expression.
Peetre proposed on page 110 of \cite{Peetre} to find out
the specific expression form of the real interpolation space
$(B^{s_0, q_0}_{p_0}, B^{s_1, q_1}_{p_1})_{\theta, r}$ when $r\neq q$.
Devore-Popov have considered the relation of K functional
between Besov spaces and the corresponding discrete spaces for frequency in Theorem 6.1 of \cite{DP}.
Lou-Yang-He-He \cite{LYHH} made progress for $s_0=s_1$ and $p_0\neq p_1$.
Besoy-Haroske-Triebel \cite{BHT} paid attention to this Peetre's longstanding open problem
and considered the cases where $p_0=q_0$, $p_1=q_1$  and $s_0-s_1= \frac{\alpha}{p_0}-\frac{\alpha}{p_1}$.
Yang-Yang-Zou-He's result \cite{YYZH}
was reported at two international conferences last year,
and they mentioned that
they used wavelets and grid K functionals
to find the part of spaces that run out of Besov type spaces,
and completely solved this longstanding open problem
on real interpolation spaces of Besov spaces.

\subsection{Real interpolation spaces for Besov type spaces}
Peetre \cite{Peetre} applied Lions-Peetre's iterative Theorem of interpolation and other skills
to get a part of real interpolation spaces of Besov type spaces.
K functional for $r=q$ can be obtained by the norm structure of the function space.

\begin{lemma} \label{lem:3.4}
{\bf See Theorem 4 of page 98 in \cite{Peetre} }
Given $s_0,s_1\in \mathbb{R}, 0<q_0,q_1\leq \infty$, $0<\theta<1$
and $\frac{1}{q}= \frac{1-\theta}{q_0} + \frac{\theta}{q_1}$.
Given $A_0, A_1$ Banach spaces.

(i) If $s_0\neq s_1$, $s= (1-\theta) s_0 + \theta s_1$ and $A_0=A_1=A$, then
$$(l^{s_0,q_0}(A), l^{s_1,q_1}(A))_{\theta, r}= l^{s,r}(A).$$

(ii) If $s_0= s_1=s$, $r=q$ and $A_0=A_1=A$, then
$$(l^{s_0,q_0}(A), l^{s_1,q_1}(A))_{\theta, r}= l^{s,r}(A).$$

(iii) If $s_0\neq s_1$, $s= (1-\theta) s_0 + \theta s_1$ and $r=q$, then
$$(l^{s_0,q_0}(A_0), l^{s_1,q_1}(A_1))_{\theta, r}= l^{s,r}((A_0,A_1)_{\theta, r}).$$
\end{lemma}

Note that $q$ is defined by $q_0,q_1$ and $\theta$, the condition $r=q$ in Lemma \ref{lem:3.4}
implies that very few indices are known for the real interpolation space.
{\bf The norm of function space only reflects the continuous structure of functions,
but K functional reflects the internal geometric structure of functions and non-linearity}.
In 1974, Cwikel \cite{C} proved that the Lions-Peetre formula for
$(l^{q_0}(A_0), l^{q_1}(A_1))_{\theta,r}$
have no reasonable generalization for any $r\neq q $.
One did not know how to deal with the case where $r\neq q$.
Peetre prosed to consider the Open problem \ref{con:1}.
See his book \cite{Peetre67book} in 1967 or
page 104 of his English translation \cite{Peetre} in 1976.

Lorentz \cite{Lorentz}
considered distribution functions and
non-increasing rearrangement functions
and introduce Lorentz spaces to extend Lebesgue spaces in the 1950s.
In 1964,
Hunt \cite{Hunt1, Hunt} extend Marcinkiewiez theorem to Lorentz spaces and
revealed the relation between the real interpolation space of Lebesgue space and Lorentz space.
Devore-Popov have considered the relation of K functional
between Besov spaces and the corresponding discrete spaces for frequency in Theorem 6.1 of \cite{DP}.
The Peetre's longstanding open problem on hierachical peudo-Banach spaces $l^{s,q}(A)$ never got off the ground
because of the lack of a way to deal with its nonlinearities.
In this paper, we consider systematically the interpolation spaces for Besov type spaces $\dot{l}^{s,q}(A)$ by wavelets
and solve Peetre's longstanding open problem at page 104 of \cite{Peetre}.

\section{Nonlinear functional calculus and Vertex functional} \label{SEC5}

In this section,
we consider the functional corresponding to the wavelet basis
and study its nonlinear structure.
Devore-Popov have considered the relation of K functional
between Besov spaces and the corresponding discrete spaces for frequency in main Theorem 6.1 of \cite{DP}.
First, we extend their theorem to wavelet K functional for general Besov hierarchical spaces.
By analyzing the wavelet coefficient characteristics of the approximation wavelet functional,
we get cuboid K functional.
It is concluded that wavelet basis is an unconditional basis of interpolation space.
Further, using the characteristics that both function spaces satisfy the quasi-trigonometric inequality,
through functional calculus, we transform cuboid K functional to vertex K functional.
Through nonlinear functional transformation,
we transform the calculation problem of functional into the classification problem of vertices.
This provides a theoretical basis for the study of topological structures on lattice points.

\subsection{From wavelet functional to cuboid functional}\label{sec:4.1}
In this subsection,
we consider a wider range of spaces than Devore-Popov's Theorem 6 in\cite{DP},
and the corresponding theorem is further accurate to the case of cuboid K functionals.
In fact, we prove that
K functional is equivalent to the infimum over the cuboid defined by the absolute value of wavelet coefficients.

Continuous K functional is the lower bound of the dynamic norm from $(\mathcal{S}', \mathcal{S}')$ to $\mathbb{R}_{+}$.
By using wavelets, we transform the dynamic norm to the discrete dynamic norm
from $(\mathbb{C}^{\dot{\Lambda}}, \mathbb{C}^{\dot{\Lambda}})$ to $\mathbb{R}_{+}$.
Thus, we transform the relative K functional to the minimal functional on grid $\dot{\Lambda}$
which takes values in $\mathbb{C}^{\dot{\Lambda}}$.
By wavelet characterization Lemma \ref{lem:cbesov} on Besov type spaces,
we clarify the Devore-Popov discretization theorem in \cite{DP}
and we obtain the following wavelet K functional.
\begin{theorem}\label{th:6.1}
Given $s_0,s_1\in \mathbb{R}, 0<p_0,p_1,q_0,q_1\leq \infty, 0<\xi<\infty$.
K functional for $f \in \dot{l}^{s_0,q_0}(A_0)+ \dot{l}^{s_1,q_1}(A_1)\in S'(\mathbb{R}^{n})$
can be transformed to the infimum of sequence $f_{j,\gamma}$
defined on $\dot{l}^{s_0,q_0}(A_0)+ \dot{l}^{s_1,q_1}(A_1)\in \mathbb{C}^{\dot{\Lambda}}$:
$$\begin{array}{l}
K (t,f, \dot{l}^{s_0,q_0}(A_0), \dot{l}^{s_1,q_1}(A_1))
=K (t,\{f_{j, \gamma}\}, \dot{l}^{s_0,q_0}(A_0), \dot{l}^{s_1,q_1}(A_1)).
\end{array}$$
\end{theorem}

\begin{proof}
For $f= (f_{j,\gamma})_{(j,\gamma)\in \dot{\Lambda}}$ and $g= (g_{j,\gamma})_{(j,\gamma)\in \dot{\Lambda}}$,
consider the norm with perturbation under threshold $t$ defined by the sum of sequence
\begin{equation}\label{eq:Kxi}
\begin{array}{rcl} && K (t,f, \dot{l}^{s_0,q_0}(A_0), \dot{l}^{s_1,q_1}(A_1), g)\\
& =&
[\sum\limits_{j\in \mathbb{Z}} 2^{j q_0 s_0}
\|\{ g_{j,\gamma}\}_{\gamma\in \dot{\Gamma}}\|_{A_0} ^{q_0} ] ^{\frac{1}{q_0}}
 + t [\sum\limits_{j\in \mathbb{Z}} 2^{j q_1 s_1}
\|\{g_{j,\gamma}-f_{j,\gamma}\}_{\gamma\in \dot{\Gamma}}\|_{A_1} ^{q_1}) ] ^{\frac{1}{q_1}}\\
& \equiv &
K (t,\{f_{j, \gamma}\}, \dot{l}^{s_0,q_0}(A_0), \dot{l}^{s_1,q_1}(A_1), \{g_{j,\gamma}\}).
\end{array}
\end{equation}
$K (t,f, \dot{l}^{s_0,q_0}(A_0), \dot{l}^{s_1,q_1}(A_1), g)$ is defined for
$f, g\in \dot{l}^{s_0,q_0}(A_0)+ \dot{l}^{s_1,q_1}(A_1)\subset S'(\mathbb{R}^{n})$.
But $K (t,\{f_{j, \gamma}\}, \dot{l}^{s_0,q_0}(A_0), \dot{l}^{s_1,q_1}(A_1), \{g_{j,\gamma}\})$ is defined for
\small{$\{f_{j,\gamma}\}_{(j,\gamma)\in \dot{\Lambda}}, \linebreak \{g_{j,\gamma}\}_{(j,\gamma)\in \dot{\Lambda}} \in \dot{l}^{s_0,q_0}(A_0)+ \dot{l}^{s_1,q_1}(A_1)\subset \mathbb{C}^{\dot{\Lambda}}.$}
Thus $K (t,f, \dot{l}^{s_0,q_0}(A_0), \dot{l}^{s_1,q_1}(A_1), g)$ maps $\dot{l}^{s_0,q_0}(A_0)+ \dot{l}^{s_1,q_1}(A_1)$ to $\mathbb{R}_{+}$
and $K (t,\{f_{j, \gamma}\}, \dot{l}^{s_0,q_0}(A_0), \dot{l}^{s_1,q_1}(A_1), \{g_{j,\gamma}\})$ maps \linebreak
$\dot{l}^{s_0,q_0}(A_0)+ \dot{l}^{s_1,q_1}(A_1)$ to $\mathbb{R}_{+}$.
The above equation \eqref{eq:Kxi} transform the study of K functional to which of sequence.
$$\begin{array}{l}
K (t,f, \dot{l}^{s_0,q_0}(A_0), \dot{l}^{s_1,q_1}(A_1)) =
\inf \limits_{g\in \dot{l}^{s_0,q_0}(A_0) } K (t,f, \dot{l}^{s_0,q_0}(A_0), \dot{l}^{s_1,q_1}(A_1), g) \\
= \inf \limits_{\{g_{j,\gamma}\}_{(j,\gamma)\in \dot{\Lambda}}\in \dot{l}^{s_0,q_0}(A_0) } K (t,\{f_{j, \gamma}\}, \dot{l}^{s_0,q_0}(A_0), \dot{l}^{s_1,q_1}(A_1), \{g_{j,\gamma}\})\\
=K (t,\{f_{j, \gamma}\}, \dot{l}^{s_0,q_0}(A_0), \dot{l}^{s_1,q_1}(A_1)).
\end{array}$$

\end{proof}

Then we analyze the wavelet coefficient characteristics of approximate wavelet functional,
we turn the K functionals over the infinite complex field
into the K functionals over the cuboid of the real field.
For $(j,\gamma)\in \dot{\Lambda}$, define
$$(Bf)_{j,\gamma}= : \left \{
\begin{array}{cl}
\frac{ f_{j,\gamma}}{|f_{j,\gamma}|}, & {\mbox \rm if\, } f_{j,\gamma}\neq 0; \\
1, & {\mbox \rm if\, } f_{j,\gamma}= 0.
\end{array}
\right. $$
Then $|(Bf)_{j,\gamma}|=1$ and $f_{j,\gamma}= |f_{j,\gamma}| (Bf)_{j,\gamma}.$
Let $$ \mathfrak{C}^{\dot{\Lambda}}_{f}= \{ \{g_{j,\gamma}\}_{(j,\gamma)\in \dot{\Lambda}}: 0\leq g_{j,\gamma}\leq |f_{j,\gamma}|\}$$
denote an infinite cuboid  in $\mathbb{R}^{\dot{\Lambda}}_{+}$.
For function $f$, using positive equivalence classes,
we transform then the K functional to the infimum of positive sequence  on infinite cuboid
$ \mathfrak{C}^{\dot{\Lambda}}_{f}$.
For $\{g_{j,\gamma}\}_{(j,\gamma)\in \dot{\Lambda}}$,
denote $\tilde{g}_{j,\gamma}= g_{j,\gamma}  (Bf)_{j,\gamma}$ and $\tilde{g}= \{\tilde{g}_{j,\gamma}\}_{(j,\gamma)\in \dot{\Lambda}}$.
Let $$K_{\mathfrak{C}}(t,f, \dot{l}^{s_0,q_0}(A_0), \dot{l}^{s_1,q_1}(A_1))=
\inf \limits_{\{g_{j,\gamma}\}_{(j,\gamma)\in \dot{\Lambda}}\in \mathfrak{C}^{\dot{\Lambda}}_{f}}
K (t,f, \dot{l}^{s_0,q_0}(A_0), \dot{l}^{s_1,q_1}(A_1), \tilde{g}).$$
$K_{\mathfrak{C}}$ denote the  functional on infinite cuboid.
The cuboid functional $K_{\mathfrak{C}}$ on infinite cuboid equals the $K$ functional.
\begin{theorem}\label{th:5.3}
For all $s_0,s_1\in \mathbb{R}$ and $0<q_0,q_1\leq \infty$,
K functional is equal to $K_{\mathfrak{C}}$ functional:
\begin{equation}\label{eq:cuboid}
\begin{array}{rcl}
K(t,f, \dot{l}^{s_0,q_0}(A_0), \dot{l}^{s_1,q_1}(A_1))
&=&K_{\mathfrak{C}}(t,f, \dot{l}^{s_0,q_0}(A_0), \dot{l}^{s_1,q_1}(A_1)).
\end{array}
\end{equation}
\end{theorem}

\begin{proof}
(i) For $g_{j,\gamma}\in \mathbb{C}^{\dot{\Lambda}}$, we have
$$|\tilde{g}_{j,\gamma}-f_{j,\gamma}|= |(\tilde{g}_{j,\gamma}  -f_{j,\gamma}) (Bf)_{j,\gamma}^{-1}|
= |\tilde{g}_{j,\gamma} (Bf)_{j,\gamma}^{-1} -f_{j,\gamma} (Bf)_{j,\gamma}^{-1}|
= |g_{j,\gamma}  -|f_{j,\gamma}||.$$
That is to say,
$$K (t,\{f_{j, \gamma}\}, \dot{l}^{s_0,q_0}(A_0), \dot{l}^{s_1,q_1}(A_1), \tilde{g}_{j,\gamma})
= K (t,|f_{j, \gamma}|, \dot{l}^{s_0,q_0}(A_0), \dot{l}^{s_1,q_1}(A_1), \{g_{j,\gamma}\} ).$$
The above equality implies that
K functional defined by the infimum of sequence $\{f_{j,\gamma}\}_{(j,\gamma)\in \dot{\Lambda}} , \{\tilde{g}_{j,\gamma} \}_{(j,\gamma)\in \dot{\Lambda}}$
defined on $\dot{l}^{s_0,q_0}(A_0)+ \dot{l}^{s_1,q_1}(A_1)\in \mathbb{C}^{\dot{\Lambda}}$
can be transformed to the K functional defined by the infimum of sequence $|f_{j,\gamma}|,g_{j,\gamma}$
where $\{|f_{j,\gamma}|\}_{(j,\gamma)\in \dot{\Lambda}}$ is defined on $\dot{l}^{s_0,q_0}(A_0)+ \dot{l}^{s_1,q_1}(A_1)\in \mathbb{R}_{+}^{\dot{\Lambda}}$
and $\{g_{j,\gamma}\}_{(j,\gamma)\in \dot{\Lambda}}$ is defined on $\dot{l}^{s_0,q_0}(A_0)+ \dot{l}^{s_1,q_1}(A_1)\in \mathbb{C}^{\dot{\Lambda}}$.

(ii) If ${\rm Im}\, g_{j,\gamma}\neq 0$, then $$|g_{j,\gamma}-|f_{j,\gamma}||
= \sqrt{||f_{j,\gamma}|- {\rm Re}\, g_{j,\gamma}|^{2} + ({\rm Im}\, g_{j,\gamma})^{2} }
> ||f_{j,\gamma}|- {\rm Re}\, g_{j,\gamma}|.$$
Hence,
$$\begin{array}{rl}
&K (t,\{|f_{j, \gamma}|\}, \dot{l}^{s_0,q_0}(A_0), \dot{l}^{s_1,q_1}(A_1), \{g_{j,\gamma}\} )\\
> & K (t,\{|f_{j, \gamma}|\}, \dot{l}^{s_0,q_0}(A_0), \dot{l}^{s_1,q_1}(A_1), \{{\rm Re}\, g_{j,\gamma}\} ).
\end{array}$$
Let $ \Re= \{ \{g_{j,\gamma}\}_{(j,\gamma)\in \dot{\Lambda}}: g_{j,\gamma}\in \mathbb{R} \}$
denote that $\{g_{j,\gamma}\}_{(j,\gamma)\in \dot{\Lambda}}$ takes value in  $\mathbb{R}^{\dot{\Lambda}}$.
Hence
\begin{equation*}
\begin{array}{rl}
& K(t,f, \dot{l}^{s_0,q_0}(A_0), \dot{l}^{s_1,q_1}(A_1))\\
= &\inf \limits_{\{g_{j,\gamma}\}_{(j,\gamma)\in \dot{\Lambda}}\in \Re }
K (t,\{|f_{j,\gamma}| \}_{(j,\gamma)\in \dot{\Lambda}}, \dot{l}^{s_0,q_0}(A_0), \dot{l}^{s_1,q_1}(A_1), \{g_{j,\gamma}\}_{(j,\gamma)\in \dot{\Lambda}}).
\end{array}
\end{equation*}
The above equality implies that
K functional defined by the infimum of $K (t, f, \dot{l}^{s_0,q_0}(A_0), \dot{l}^{s_1,q_1}(A_1), \tilde{g}_{j,\gamma})$
where the sequence $\{f_{j,\gamma}\}_{(j,\gamma)\in \dot{\Lambda}}, \{\tilde{g}_{j,\gamma}\}_{(j,\gamma)\in \dot{\Lambda}}$
take values on $\dot{l}^{s_0,q_0}(A_0)  + \dot{l}^{s_1,q_1}(A_1)\in \mathbb{C}^{\dot{\Lambda}}$
can be transformed to the K functional defined by the infimum of sequence $\{|f_{j,\gamma}|\}_{(j,\gamma)\in \dot{\Lambda}}, \{g_{j,\gamma}\}_{(j,\gamma)\in \dot{\Lambda}}$
where $\{ |f_{j,\gamma}| \}_{(j,\gamma)\in \dot{\Lambda}} $ is defined on  $\dot{l}^{s_0,q_0}(A_0)+ \dot{l}^{s_1,q_1}(A_1)\in \mathbb{R}_{+}^{\dot{\Lambda}}$
and $\{g_{j,\gamma}\}_{(j,\gamma)\in \dot{\Lambda}} $ is defined on $\dot{l}^{s_0,q_0}(A_0)+ \dot{l}^{s_1,q_1}(A_1)\in \mathbb{R}^{\dot{\Lambda}}$.

(iii) Further, if $g_{j,\gamma}<0$, then
$$|f_{j,\gamma}|- g_{j,\gamma}=|f_{j,\gamma}|+|g_{j,\gamma}|> |f_{j,\gamma}|- |g_{j,\gamma}|.$$
Hence,
$$\begin{array}{rl}
& K (t,\{|f_{j, \gamma}|\}, \dot{l}^{s_0,q_0}(A_0), \dot{l}^{s_1,q_1}(A_1), \{g_{j,\gamma}\} )\\
> & K (t,\{|f_{j, \gamma}|\}, \dot{l}^{s_0,q_0}(A_0), \dot{l}^{s_1,q_1}(A_1), \{|g_{j,\gamma}|\} ).
\end{array}$$

Let $ \mathcal{P} = \{ \{g_{j,\gamma}\}_{(j,\gamma)\in \dot{\Lambda}}: g_{j,\gamma}\geq 0 \}$
denote that $\{g_{j,\gamma}\}_{(j,\gamma)\in \dot{\Lambda}}$ takes value in  $\mathbb{R}_{+}^{\dot{\Lambda}}$.
Hence
\begin{equation*}
\begin{array}{rl}
& K(t,f, \dot{l}^{s_0,q_0}(A_0), \dot{l}^{s_1,q_1}(A_1))\\
= &\inf \limits_{\{g_{j,\gamma}\}_{(j,\gamma)\in \dot{\Lambda}}\in  \mathcal{P} }
K (t,\{|f_{j,\gamma}| \}_{(j,\gamma)\in \dot{\Lambda}}, \dot{l}^{s_0,q_0}(A_0), \dot{l}^{s_1,q_1}(A_1), \{g_{j,\gamma}\}_{(j,\gamma)\in \dot{\Lambda}}).
\end{array}
\end{equation*}
The above equality implies that
the relative K functional is
transformed to the K functional defined by the infimum of sequence $\{|f_{j,\gamma}|\}_{(j,\gamma)\in \dot{\Lambda}}, \{g_{j,\gamma}\}_{(j,\gamma)\in \dot{\Lambda}}$
where $\{ |f_{j,\gamma}| \}_{(j,\gamma)\in \dot{\Lambda}} $ is defined on  $\dot{l}^{s_0,q_0}(A_0)+ \dot{l}^{s_1,q_1}(A_1)\in \mathbb{R}_{+}^{\dot{\Lambda}}$
and $\{g_{j,\gamma}\}_{(j,\gamma)\in \dot{\Lambda}} $ is defined on $\dot{l}^{s_0,q_0}(A_0)+ \dot{l}^{s_1,q_1}(A_1)\in \mathbb{R}_{+}^{\dot{\Lambda}}$.

(iv) For $\{g_{j,\gamma}\}_{(j,\gamma)\in \dot{\Lambda}} $ is defined on $\dot{l}^{s_0,q_0}(A_0)+ \dot{l}^{s_1,q_1}(A_1)\in \mathbb{R}^{\dot{\Lambda}}_{+}$,
we construct $\{g^*_{j,\gamma} \}_{(j,\gamma)\in \dot{\Lambda}}\in \mathfrak{C}^{\dot{\Lambda}}_{f}$
such that
\begin{equation}\label{eq:mathfrak}
\begin{array}{rl}
&K (t,\{|f_{j, \gamma}|\}, \dot{l}^{s_0,q_0}(A_0), \dot{l}^{s_1,q_1}(A_1), \{g_{j,\gamma}\} )\\
\geq & K (t,\{|f_{j, \gamma}|\}, \dot{l}^{s_0,q_0}(A_0), \dot{l}^{s_1,q_1}(A_1), \{g^*_{j,\gamma}\} ).
\end{array}
\end{equation}
To prove the above inequality, we consider three cases
for $\{g_{j,\gamma}\}_{(j,\gamma)\in \dot{\Lambda}}\in \mathcal{P} $.
(A). If $g_{j,\gamma}\geq 2|f_{j,\gamma}|$, then choose $g_{j,\gamma}^{*}=0$, we have
$$g_{j,\gamma}>g^*_{j,\gamma} \mbox { and } ||f_{j,\gamma}|- g_{j,\gamma}|\geq |f_{j,\gamma}|= |f_{j,\gamma}|- g^*_{j,\gamma}.$$
(B) If $|f_{j,\gamma}|< g_{j,\gamma}<2|f_{j,\gamma}|$, then there exists $g^{*}_{j,\gamma}$ satisfying that
(i) $0<g^{*}_{j,\gamma}< |f_{j,\gamma}|$
and  (ii) $g_{j,\gamma}+g^{*}_{j,\gamma}=2|f_{j,\gamma}|$. Hence
$$g_{j,\gamma}>g^*_{j,\gamma} \mbox { and }
||f_{j,\gamma}|- g_{j,\gamma}|=g_{j,\gamma}-|f_{j,\gamma}| =|f_{j,\gamma}|-g^{*}_{j,\gamma}.$$
(C) If $0< g_{j,\gamma}\leq |f_{j,\gamma}|$,
then choose $g_{j,\gamma}^{*}=g_{j,\gamma}.$

For $\{g_{j,\gamma}^{*}\}_{(j,\gamma)\in \dot{\Lambda}}$ constructed in the above three cases (A), (B) and (C),
we know that,
$\forall (j,\gamma)\in \dot{\Lambda}$, $g^{*}_{j,\gamma}\geq 0$ and $|f_{j,\gamma}|- g^{*}_{j,\gamma}\geq 0$.
Hence
$\{g^{*}_{j,\gamma}\}_{(j,\gamma)\in \dot{\Lambda}}\in  \mathfrak{C}^{\dot{\Lambda}}_{f}$.
Further,
$\{g_{j,\gamma}^{*}\}_{(j,\gamma)\in \dot{\Lambda}}$
satisfies \eqref{eq:mathfrak}.
Hence
$$\begin{array}{rl}
&\inf \limits_{\{g_{j,\gamma}\}_{(j,\gamma)\in \dot{\Lambda}}\in  \mathcal{P} } K (t,\{|f_{j, \gamma}|\}, \dot{l}^{s_0,q_0}(A_0), \dot{l}^{s_1,q_1}(A_1), \{g_{j,\gamma}\} )\\
= & \inf \limits_{\{g^{*}_{j,\gamma}\}_{(j,\gamma)\in \dot{\Lambda}}\in  \mathfrak{C}^{\dot{\Lambda}}_{f} } K (t,\{|f_{j, \gamma}|\}, \dot{l}^{s_0,q_0}(A_0), \dot{l}^{s_1,q_1}(A_1), \{g^*_{j,\gamma}\} ).
\end{array}$$
The above equality implies that
K functional defined by the infimum of sequence $\{f_{j,\gamma}\}_{(j,\gamma)\in \dot{\Lambda}}, \{g_{j,\gamma}\}_{(j,\gamma)\in \dot{\Lambda}}$
defined on $\dot{l}^{s_0,q_0}(A_0)+ \dot{l}^{s_1,q_1}(A_1)\in \mathbb{R}_{+}^{\dot{\Lambda}}$
can be transformed to the K functional defined by the infimum of sequence $\{|f_{j,\gamma}|\}_{(j,\gamma)\in \dot{\Lambda}}, \{g^{*}_{j,\gamma}\}_{(j,\gamma)\in \dot{\Lambda}}$
where $\{|f_{j,\gamma}|\}_{(j,\gamma)\in \dot{\Lambda}}$ is defined on $\dot{l}^{s_0,q_0}(A_0)+ \dot{l}^{s_1,q_1}(A_1)\in \mathbb{R}_{+}^{\dot{\Lambda}}$
and $\{g^{*}_{j,\gamma}\}_{(j,\gamma)\in \dot{\Lambda}}\in  \mathfrak{C}^{\dot{\Lambda}}_{f}$.
%and for $(j,\gamma)\in \dot{\Lambda}$, $g_{j,\gamma}\geq 0$ and $|f_{j,\gamma}|- g_{j,\gamma}\geq 0$.
%Hence we need consider $\{g_{j,\gamma}\}_{(j,\gamma)\in \dot{\Lambda}}$ only defined on the infinite cuboid $\mathfrak{C}^{\dot{\Lambda}}_{f}$.
%Using positive equivalence classes and

By the above four steps,
we have transformed the K functional to the infimum of positive sequence  on infinite cuboid
$ \mathfrak{C}^{\dot{\Lambda}}_{f}$.
Hence we get \eqref{eq:cuboid}.

\end{proof}

The following theorem is a corollary of Theorem \ref{th:5.3}
and implies that wavelet basis are unconditional basis for relative real interpolation spaces.
\begin{theorem}\label{th:absolute}
Given $s_0,s_1\in \mathbb{R}, 0<p_0,p_1,q_0,q_1\leq \infty, 0<\xi<\infty$.
If $|f_{j,\gamma}| \leq |g_{j,\gamma}|, \forall (j,\gamma)\in \Lambda$, then
$$K (t,f, \dot{l}^{s_0,q_0}(A_0), \dot{l}^{s_1,q_1}(A_1))\leq K (t,g, \dot{l}^{s_0,q_0}(A_0), \dot{l}^{s_1,q_1}(A_1)).$$
\end{theorem}

The above Theorem \ref{th:absolute} tells us the K functional is determined by the absolute value of the wavelet coefficients.
Wavelet basis are unconditional basis for real interpolation spaces $(\dot{l}^{s_0,q_0}(A_0), \dot{l}^{s_1,q_1}(A_1))_{\theta,r}$.
Hence denote $\tilde{f}$ the function whose wavelet coefficients
are equal to the absolute value of the wavelet coefficients of $f$ and
$\tilde{f}_{j,\gamma}=|f_{j,\gamma}|$.
According to the above Theorems \ref{th:5.3} and \ref{th:absolute},
we only need to compute discrete K functionals with positive wavelet coefficients.
So let's assume that all the wavelet coefficients are positive.

\subsection{Vertex functionals} \label{SEC6}
We use quasi trigonometric inequalities to consider the cuboid functional
and convert functional calculation into vertex classification.
Given $s_0,s_1\in \mathbb{R}, 0<q_0,q_1\leq \infty$.
Note that, by wavelet characterization,
\begin{equation*}
\begin{array}{rl}
&K (t,f, \dot{l}^{s_0,q_0}(A_0), \dot{l}^{s_1,q_1}(A_1), f_0)\\
=& [\sum\limits_{j\in \mathbb{Z}} 2^{j q_0 s_0}
\|\{ f^{0}_{j,\gamma}- f_{j,\gamma}\}_{\gamma\in \dot{\Gamma}}\|_{A_0} ^{q_0} ] ^{\frac{1}{q_0}}
+ t [\sum\limits_{j\in \mathbb{Z}} 2^{j q_1 s_1}
\|\{f^{0}_{j,\gamma}\}_{\gamma\in \dot{\Gamma}}\|_{A_1} ^{q_1}) ] ^{\frac{1}{q_1}}.
\end{array}
\end{equation*}
By Theorem \ref{eq:cuboid}, we have
$$K(t,f, \dot{l}^{s_0,q_0}(A_0), \dot{l}^{s_1,q_1}(A_1))= K_{\mathfrak{C}}(t,\tilde{f}, \dot{l}^{s_0,q_0}(A_0), \dot{l}^{s_1,q_1}(A_1)).$$

Let $ V^{\dot{\Lambda}}_{f} = \{(g_{j,\gamma})_{(j,\gamma)\in \dot{\Lambda}}: g_{j,\gamma}=0 {\, \rm or \, } |f_{j,\gamma}|\}$
be the vertex point set of the cuboid
$ \mathfrak{C}^{\dot{\Lambda}}_{f}= \{(g_{j,\gamma})_{(j,\gamma)\in \dot{\Lambda}}: 0\leq g_{j,\gamma}\leq |f_{j,\gamma}|\}$.
We convert the $K_{\mathfrak{C}}$ functional on the cuboid $ \mathfrak{C}^{\dot{\Lambda}}_{f}$ to the vertex case on $ V^{\dot{\Lambda}}_{f}$.
For $\tilde{f}= \{|f_{j,\gamma}|\}_{ (j,\gamma)\in \dot{\Lambda} }$, denote
\small{
$$\begin{array}{cl}
&K_{V} (t,f, \dot{l}^{s_0,q_0}(A_0), \dot{l}^{s_1,q_1}(A_1))\\
= & \inf \limits_{\{f^{0}_{j,\gamma}\}_{(j,\gamma)\in \dot{\Lambda}}\in V^{\dot{\Lambda}}_{f} } K_{\mathfrak{C}}(t,\tilde{f}, \dot{l}^{s_0,q_0}(A_0), \dot{l}^{s_1,q_1}(A_1), f_0).
\end{array}$$}
The vertex functionals $K_{V}$ is equivalent to cuboid functionals $K_{\mathfrak{C}}$.
\begin{theorem}\label{th:vertex}
Given $s_0,s_1\in \mathbb{R}, 0<q_0,q_1\leq \infty$. We have
$$K _{\mathfrak{C}} (t,f, \dot{l}^{s_0,q_0}(A_0), \dot{l}^{s_1,q_1}(A_1)) \sim  K_{V} (t,f, \dot{l}^{s_0,q_0}(A_0), \dot{l}^{s_1,q_1}(A_1)).$$
That is to say,\\
(i) $K _{\mathfrak{C}} (t,f, \dot{l}^{s_0,q_0}(A_0), \dot{l}^{s_1,q_1}(A_1)) \leq K_{V} (t,f, \dot{l}^{s_0,q_0}(A_0), \dot{l}^{s_1,q_1}(A_1))$\\
(ii) $K_{V} (t,f, \dot{l}^{s_0,q_0}(A_0), \dot{l}^{s_1,q_1}(A_1)) \lesssim K_{\mathfrak{C}} (t,f, \dot{l}^{s_0,q_0}(A_0), \dot{l}^{s_1,q_1}(A_1))$.
\end{theorem}

\begin{proof}
(1) $K _{\mathfrak{C}} (t,f, \dot{l}^{s_0,q_0}(A_0), \dot{l}^{s_1,q_1}(A_1))$ takes infimum for all possible $f_0$ in the cuboid,
but $K _{V} (t,f, \dot{l}^{s_0,q_0}(A_0), \dot{l}^{s_1,q_1}(A_1))$ takes infimum for all possible $f_0$ in the vertex of the cuboid.
Hence (i) is true
$$K _{\mathfrak{C}}(t,f, \dot{l}^{s_0,q_0}(A_0), \dot{l}^{s_1,q_1}(A_1)) \leq K _{V} (t,f, \dot{l}^{s_0,q_0}(A_0),  \linebreak \dot{l}^{s_1,q_1}(A_1)).$$

(2) To prove (ii), according to Theorem \ref{th:5.3}, we need only consider $\tilde{f}= \{|f_{j,\gamma}|\}_{ (j,\gamma)\in \dot{\Lambda} }$,
and $g=\{g_{j,\gamma}\}_{ (j,\gamma)\in \dot{\Lambda} }\in \mathfrak{C}^{\dot{\Lambda}}_{\tilde{f}}$.
In fact, we find
$g^{*}\in V^{\dot{\Lambda}}_{f}$ and prove that the following inequality is true:
\begin{equation}\label{eq:CVG}
K (t,\tilde{f}, \dot{l}^{s_0,q_0}(A_0), \dot{l}^{s_1,q_1}(A_1), g^{*}) \lesssim
K (t,\tilde{f}, \dot{l}^{s_0,q_0}(A_0), \dot{l}^{s_1,q_1}(A_1), g).
\end{equation}

But for all $\tilde{f}_{j,\gamma}= |f_{j,\gamma}|>0$ and $g\in \mathfrak{C}^{\dot{\Lambda}}_{\tilde{f}}$, we must have
$$g_{j,\gamma}\leq \frac{1}{2} |f_{j,\gamma}| {\mbox\,\, or \,\,} g_{j,\gamma}> \frac{1}{2} |f_{j,\gamma}|.$$
For the first case, we denote $(j,\gamma)\in \Gamma_{1,j}$ and take $g^{*}_{j,\gamma}=0$.
For the second case, we denote $(j,\gamma)\in \Gamma_{2,j}$ and take $g^{*}_{j,\gamma}=|f_{j,\gamma}|$.
We get $\{\gamma\in \dot{\Gamma}_{j}, |f_{j,\gamma}|\neq 0\}= \Gamma_{1,j} \bigcup \Gamma_{2,j}$ and
$g^{*}= \{g^{*}_{j,\gamma}\}_{ (j,\gamma)\in \dot{\Lambda} } $. Denote
$$\begin{array}{c}
 K (t,\tilde{f}, \dot{l}^{s_0,q_0}(A_0), \dot{l}^{s_1,q_1}(A_1), g^{*})\\
\equiv :
[\sum\limits_{j\in \mathbb{Z}} 2^{j q_0 s_0}
\|\{|f_{j,\gamma}|\}_{\gamma\in \Gamma_{1,j}}\|_{A_0} ^{q_0} ] ^{\frac{1}{q_0}}
+ t [\sum\limits_{j\in \mathbb{Z}} 2^{j q_1 s_1}
\|\{g^{*}_{j,\gamma}\}_{\gamma\in \Gamma_{2,j}}\|_{A_1} ^{q_1} ] ^{\frac{1}{q_1}}.
\end{array}$$

Further,
$$\|\{|f_{j,\gamma}|\}_{\gamma\in \Gamma_{1,j}}\|_{A_0}
\lesssim \|\{|f_{j,\gamma}|-g_{j,\gamma}\}_{\gamma\in \Gamma_{1,j}}\|_{A_0},$$
$$\|\{g_{j,\gamma}\}_{\gamma\in \Gamma_{2,j}}\|_{A_1}
\lesssim \|\{g_{j,\gamma}\}_{\gamma\in \Gamma_{2,j}}\|_{A_1} .$$
Hence we have
$$\begin{array}{l}  K (t, \tilde{f}, \dot{l}^{s_0,q_0}(A_0), \dot{l}^{s_1,q_1}(A_1), g^{*}) \\
\lesssim  [\sum\limits_{j\in \mathbb{Z}} 2^{j q_0 s_0}
\|\{ (f_{j,\gamma}-g^{*}_{j,\gamma})\}_{\gamma\in \Gamma_{1,j}}\|_{A_0}^{q_0} ] ^{\frac{1}{q_0}}
+  t [\sum\limits_{j\in \mathbb{Z}} 2^{j q_1 s_1}
\|\{g^{*}_{j,\gamma}\}_{\gamma\in \Gamma_{2,j}}\|_{A_1}^{q_1} ] ^{\frac{1}{q_1}}\\
\lesssim  [\sum\limits_{j\in \mathbb{Z}} 2^{j q_0 s_0}
\|\{ (f_{j,\gamma}-g_{j,\gamma})\}_{\gamma\in \dot{\Gamma}}\|_{A_0}^{q_0} ] ^{\frac{1}{q_0}}
+  t [\sum\limits_{j\in \mathbb{Z}} 2^{j q_1 s_1 }
\|\{g_{j,\gamma}\}_{\gamma\in \dot{\Gamma}}\|_{A_1}^{q_1} ] ^{\frac{1}{q_1}}\\
\lesssim   K (t,\tilde{f}, \dot{l}^{s_0,q_0}(A_0), \dot{l}^{s_1,q_1}(A_1), g).
\end{array}$$
Hence the Equation \eqref{eq:CVG} is true.
\end{proof}

Theorem \ref{th:vertex} allows us to classify the index set $\dot{\Lambda}$
and provides a possibility to consider K functional by vertexes classification.
The calculation of K functional is turned into the study of the topology of lattice points.
Next we will study the relationship of topological structures on three types of lattice
based on vertex K functional:
full grid $\dot{\Lambda}$, main grid $\mathbb{Z}$ and layer grid $\dot{\Gamma}_{j}$.

\section{K functional for $A_0=A_1$ or $q_0=q_1$}\label{SEC7}
In this section, we study the structure of lattice topology for $A_0=A_1$ or $q_0=q_1$.
We assume that $A_0$ and $A_1$ both satisfy the absolute value property \eqref{eq:absolute}.

\subsection{K functional for $A_0=A_1$} \label{SEC:5.1}
The topology of the functions on the layer grid $\dot{\Gamma}_{j}$ is the same.
So we only need to consider the functional of
discrete weighted Lebesgue space $\dot{l}^{s,q}(\mathbb{Z})$  defined on the main grid $\mathbb{Z}$.
\begin{definition}
Given $s\in \mathbb{R}, 0<q\leq \infty$.
Denote $F=(F_{j})_{j\in \mathbb{Z}} \in \dot{l}^{s,q}(\mathbb{Z}) = \dot{l}^{s,q}$, if
$$  [\sum\limits_{j\in \mathbb{Z}} 2^{j s_0 q_0} F_{j}^{q_0} ]^{\frac{1}{q_0}}<\infty.$$
\end{definition}

{\bf Firstly, let's briefly mention K-functionals on $\mathbb{Z}$.}
For $F=(F_{j})_{j\in \mathbb{Z}}$, denote
$$\begin{array}{c}
K(t,F, \dot{l}^{s_0,q_0}(\mathbb{Z}), \dot{l}^{s_1,q_1}(\mathbb{Z}), x) =K(t,F, \dot{l}^{s_0,q_0}, \dot{l}^{s_1,q_1}, x)\\
=  [\sum\limits_{j\in \mathbb{Z}} 2^{j s_0 q_0} x_{j}^{q_0} ]^{\frac{1}{q_0}}
+ t [\sum\limits_{j\in\mathbb{Z}} 2^{js_1 q_1} (F_{j}-x_{j})^{q_1} ] ^{\frac{1}{q_1}}.\\
K(t,F, \dot{l}^{s_0,q_0}, \dot{l}^{s_1,q_1})= \inf\limits_{x\in \dot{l}^{s_0,q_0}} K(t,F, \dot{l}^{s_0,q_0}, \dot{l}^{s_1,q_1}, x).
\end{array}$$

The following Remark tells us that the quantity $K(t,F, \dot{l}^{s_0,q_0}, \dot{l}^{s_1,q_1})$ is known.
\begin{remark} \label{re:lsq}
According to two cases (i) $q_0\neq q_1$ and (ii) $q_0=q_1$ and $s_0\neq s_1$,
Yang-Yang-Zou-He \cite{YYZH} has obtained the expression of the K functional $K(t,F, \dot{l}^{s_0,q_0}, \dot{l}^{s_1,q_1})$.
In fact, similar to the discretization skills for Lebesgue spaces $(L^{q_0}, L^{q_1})$ in Hunt's work \cite{Hunt}
or similar to the skills for the Besov spaces $(B^{s_0,q_0}_{p}, B^{s_1,q_1}_{p})$ in Peetre's book \cite{Peetre},
we can also get easily the concrete expression of $K(t,F, \dot{l}^{s_0,q_0}, \dot{l}^{s_1,q_1})$.
\end{remark}

Further, we introduce $K_{\mathfrak{C} }$ functional and $K_{V}$ functional
$$\begin{array}{l}K_{\mathfrak{C} }(t,F, \dot{l}^{s_0,q_0}, \dot{l}^{s_1,q_1})=\inf\limits_{0\leq x_{j}\leq F_{j}} K(t,F, \dot{l}^{s_0,q_0}, \dot{l}^{s_1,q_1}, x)\\
K_{V}(t,F, \dot{l}^{s_0,q_0}, \dot{l}^{s_1,q_1})= \inf\limits_{ x_{j}=0 \mbox {\, or\, } F_{j}} K(t,F, \dot{l}^{s_0,q_0}, \dot{l}^{s_1,q_1}, x).
\end{array}$$

Similar to what we did in the Section \ref{SEC5}, we have
\begin{lemma}\label{lem:555}
$$\begin{array}{l}K(t,F, \dot{l}^{s_0,q_0}, \dot{l}^{s_1,q_1})=K_{\mathfrak{C} }(t,F, \dot{l}^{s_0,q_0}, \dot{l}^{s_1,q_1})
\sim  K_{V}(t,F, \dot{l}^{s_0,q_0}, \dot{l}^{s_1,q_1}).
\end{array}$$
\end{lemma}

We apply then Remark \ref{re:lsq} and Lemma \ref{lem:555} to consider K functionals for $A_0=A_1=A$.
In fact, we proved that, for such case,
we didn't have to deal with the internal topology of the layer grid $\dot{\Gamma}_{j}$.
Denote
\begin{equation} \label{eq:5.10000}
F_{j}= \|f_{j,\gamma}\|_{A (\dot{\Gamma}_{j})} \mbox{ and } F=\{F_{j}\}_{j\in \mathbb{Z}}.
\end{equation}
We use K functionals of different vertices as transformations several times
to express the corresponding K functionals as K functionals of sequence $F$.

Let $\Omega_{j}=\{\omega_{j}, \omega_{j}\subset \dot{\Gamma}_{j}\}$ be the set of subset of $\dot{\Gamma}_{j}$.
For $\omega_{j}\in \Omega_{j}$, denote $\omega^{c}_{j}= \dot{\Gamma}_{j}\backslash \omega_{j}$.
For $\omega_{j}\in \Omega_{j}$, let
$X_{j,\omega_{j}}= \|f_{j,\gamma}\|_{A(\omega_{j})}$ and $ Y_{j,\omega_{j}}= \|f_{j,\gamma}\|_{A (\omega^{c}_{j})}$.
We have
\begin{equation}\label{eq:AB0}
0\leq X_{j,\omega_{j}}\leq F_{j} \, \mbox{ \, and \,} \,0\leq Y_{j,\omega_{j}}\leq F_{j} \mbox{\,\, and}
\end{equation}
\begin{equation}\label{eq:ABC0}
F_{j} \leq X_{j,\omega_{j}} + Y_{j,\omega_{j}}\leq  2F_{j}.
\end{equation}

Denote $\Lambda_0=\{(j,\gamma): j\in \mathbb{Z}, \gamma\in \omega_{j}\}$.
Denote
$$K(t,F, \dot{l}^{s_0,q_0}, \dot{l}^{s_1,q_1}, x) = [\sum\limits_{j\in \mathbb{Z}} 2^{j s_0 q_0} x_{j}^{q_0} ]^{\frac{1}{q_0}}
+ t [\sum\limits_{j\in\mathbb{Z}} 2^{js_1 q_1} (F_{j}-x_{j})^{q_1} ] ^{\frac{1}{q_1}}.$$

Because of vertex functionals,
we can write the K functional $K(t,f, \dot{l}^{s_0,q_0}(A), \dot{l}^{s_1,q_1}(A))$
as K functional $K(t,F, \dot{l}^{s_0,q_0}, \dot{l}^{s_1,q_1}).$
By Remark \ref{re:lsq}, the following Theorem has given a concrete expression of K functional for the case $A_0=A_1$.
\begin{theorem} \label{th:5.1}
Given $s_0,s_1\in \mathbb{R}, 0<q_0,q_1\leq \infty$.
K functional $K(t,f, \dot{l}^{s_0,q_0}(A), \dot{l}^{s_1,q_1}(A))$ can be characterized by the known K functional
%$$K(t,F, \dot{l}^{s_0,q_0}(\mathbb{Z}), \dot{l}^{s_1,q_1}(\mathbb{Z}))=$$
$K(t,F, \dot{l}^{s_0,q_0}, \dot{l}^{s_1,q_1})$
where $F$ defined in the equation \eqref{eq:5.10000} by wavelet coefficients:
$$K(t,f, \dot{l}^{s_0,q_0}(A), \dot{l}^{s_1,q_1}(A)) \sim K(t,F, \dot{l}^{s_0,q_0}, \dot{l}^{s_1,q_1}).$$
\end{theorem}

\begin{proof}
For each set $\Lambda_0$,
denote $$K(t,f, \dot{l}^{s_0,q_0}(A), \dot{l}^{s_1,q_1}(A), \Lambda_0)
= [\sum\limits_{j\in \mathbb{Z}} 2^{j s_0 q_0} X_{j,\omega_{j}}^{q_0} ]^{\frac{1}{q_0}}
+ t [\sum\limits_{j\in\mathbb{Z}} 2^{js_1 q_1} Y_{j,\omega_{j}}^{q_1} ] ^{\frac{1}{q_1}}.  $$
Hence K functional can be written as follows:
$$K(t,f, \dot{l}^{s_0,q_0}(A), \dot{l}^{s_1,q_1}(A))
=\inf\limits_{\Lambda_0\subset \dot{\Lambda}}
K(t,f, \dot{l}^{s_0,q_0}(A), \dot{l}^{s_1,q_1}(A), \Lambda_0).$$

(i) On one hand, by definition,
$$\begin{array}{l}K(t,f, \dot{l}^{s_0,q_0}(A), \dot{l}^{s_1,q_1}(A))
=\inf\limits_{\Lambda_0\subset \dot{\Lambda}}
K(t,f, \dot{l}^{s_0,q_0}(A), \dot{l}^{s_1,q_1}(A), \Lambda_0)\\
\geq \inf\limits_{F_{j} \leq X_{j,\omega_{j}} + Y_{j,\omega_{j}}\leq  2F_{j}}
\{  [\sum\limits_{j\in \mathbb{Z}} 2^{j s_0 q_0} X_{j,\omega_{j}}^{q_0} ]^{\frac{1}{q_0}}
+ t [\sum\limits_{j\in\mathbb{Z}} 2^{js_1 q_1} Y_{j,\omega_{j}}^{q_1} ] ^{\frac{1}{q_1}}\}\\
\sim \inf\limits_{ X_{j,\omega_{j}} + Y_{j,\omega_{j}}=F_{j}}
\{  [\sum\limits_{j\in \mathbb{Z}} 2^{j s_0 q_0} X_{j,\omega_{j}}^{q_0} ]^{\frac{1}{q_0}}
+ t [\sum\limits_{j\in\mathbb{Z}} 2^{js_1 q_1} Y_{j,\omega_{j}}^{q_1} ] ^{\frac{1}{q_1}}\}\\
\sim K_{V}(t,F, \dot{l}^{s_0,q_0}, \dot{l}^{s_1,q_1})\sim K(t,F, \dot{l}^{s_0,q_0}, \dot{l}^{s_1,q_1}).
\end{array}$$

(ii) On other hand,
\begin{equation*}
\begin{aligned}
%\begin{array}{l}
&K(t,F, \dot{l}^{s_0,q_0}, \dot{l}^{s_1,q_1})\sim K_{V}(t,F, \dot{l}^{s_0,q_0}, \dot{l}^{s_1,q_1})\\
&\sim \inf \limits_{ X_{j,\omega_{j}} \mbox{\, or \,} Y_{j,\omega_{j}}=0}
\{  [\sum\limits_{j\in \mathbb{Z}} 2^{j s_0 q_0} X_{j,\omega_{j}}^{q_0} ]^{\frac{1}{q_0}}
+ t [\sum\limits_{j\in\mathbb{Z}} 2^{js_1 q_1} Y_{j,\omega_{j}}^{q_1} ] ^{\frac{1}{q_1}}\}\\
& \geq  K(t,f, \dot{l}^{s_0,q_0}(A), \dot{l}^{s_1,q_1}(A))\\
&\geq \inf\limits_{\Lambda_0\subset \dot{\Lambda}}
K(t,f, \dot{l}^{s_0,q_0}(A), \dot{l}^{s_1,q_1}(A), \Lambda_0)
.
%\end{array}
\end{aligned}
\end{equation*}

\end{proof}

\subsection{K functional for $q_0=q_1$} \label{SEC:5.2}
Then we consider the case of $q_0=q_1$.
We use both the structure of the main grid $\mathbb{Z}$
and the following quantities $X_j$ and $Y_j$ related to the K functional on the layer grid $\dot{\Gamma}_{j}$
to give the equivalent K-functional on the whole grid $\dot{\Lambda}$.
{\bf By vertex functionals} Theorem \ref{th:vertex},
$K_{V}(t, \{f_{j,\gamma}\}_{\gamma\in \dot{\Gamma}_{j}}, A_0, A_1)$ is known.
\begin{corollary}
$\forall j\in \mathbb{Z}, \{f_{j,\gamma}\}_{\gamma\in \dot{\Gamma}_{j}}$ and $0<t<\infty$,
there exists $\Gamma_{j}(t)\subset \dot{\Gamma}_{j}$ such that
\begin{equation} \label{eq:5.50000}
X_{j}=X_{j}(t) = \|f_{j,\gamma}\|_{A_{0}(\Gamma_{j}(t))}
\mbox{ and } Y_{j}=Y_{j}(t) = \|f_{j,\gamma}\|_{A_{1}(\dot{\Gamma}_{j}\backslash \Gamma_{j}(t))}
\end{equation}
and
\begin{equation} \label{eq:5.40000}
\begin{array}{l}  K_{V}(t, \{f_{j,\gamma}\}_{\gamma\in \dot{\Gamma}_{j}}, A_0, A_1)
=X_{j}(t) +  t Y_{j}(t).
\end{array}
\end{equation}

\end{corollary}

We call $X_{j}(t)$ and $Y_{j}(t)$ the contributions
corresponding to $A_0$ and $A_1$ of the vertex K functional on the layer grid $\Gamma_j$.
Based on the above $X_{j}(t)$ and $Y_{j}(t)$,
we get $K (t, f, \dot{l}^{s_0,q}(A_0), \dot{l}^{s_1,q}(A_1))$ by the following composition Theorem:
\begin{theorem} \label{th:5.2}
Given $s_0,s_1\in \mathbb{R}, 0<q_0=q_1= q \leq \infty.$
For $X_{j}$ and $Y_{j}$ defined by the known $K_{V}(t, \{f_{j,\gamma}\}_{\gamma\in \dot{\Gamma}_{j}}, A_0, A_1)$
in the equations \eqref{eq:5.40000} and \eqref{eq:5.50000}, we have

\begin{equation}\label{eq:10.3}
\begin{array}{rl}
&K (t, f, \dot{l}^{s_0,q}(A_0), \dot{l}^{s_1,q}(A_1))\\
\sim &\{\sum\limits_{j\in \mathbb{Z}} [ 2^{j s_0 } ( X_j
+ t  2^{j (s_1-s_0) }  Y_j) ]^{q}\}^{\frac{1}{q}}\\
\sim &\{\sum\limits_{j\in \mathbb{Z}} 2^{j s_0 q} X_j^{q}
+ t^{q} \sum\limits_{j\in \mathbb{Z}} 2^{j s_1 q}  Y_j^{q}\}^{\frac{1}{q}}\\
\sim &\{\sum\limits_{j\in \mathbb{Z}} 2^{j s_0 q} X_j^{q}\}^{\frac{1}{q}}
+ t  \{\sum\limits_{j\in \mathbb{Z}} 2^{j s_1 q}  Y_j^{q}\}^{\frac{1}{q}}.
\end{array}
\end{equation}
\end{theorem}

\begin{proof}

Because of vertex functionals, we can consider all subsets $\Omega_{j}$ of $\dot{\Gamma}_{j}$.
Denote $\Lambda_0= \{(j,\gamma)\in \dot{\Lambda}, \gamma\in \Omega_{j}\subset \Gamma_{j}\}$.
If $q_0=q_1=q$, then %by Theorems \ref{th:5.3} and \ref{th:vertex},
$$\begin{array}{rl} &K (t, f, \dot{l}^{s_0,q}(A_0), \dot{l}^{s_1,q}(A_1), \Lambda_0 )\\
= &\{\sum\limits_{j\in \mathbb{Z}} 2^{j q s_0} \|f_{j,\gamma}\|_{A_0(\Omega_{j}) }^{q}\}^{\frac{1}{q}}
+ t \{\sum\limits_{j\in \mathbb{Z}} 2^{j s_1 q} \|f_{j,\gamma}\|_{A_1 (\dot{\Gamma}_{j}\backslash \Omega_{j}) }^{q}\}^{\frac{1}{q}}.
\end{array}$$

It is easy to see
$$\begin{array}{rl} &K (t, f, \dot{l}^{s_0,q}(A_0), \dot{l}^{s_1,q}(A_1), \Lambda_0 )\\
\sim &\{\sum\limits_{j\in \mathbb{Z}} 2^{j q s_0} \|f_{j,\gamma}\|_{A_0(\Omega_{j} ) }^{q}
+ t^{q} \sum\limits_{j\in \mathbb{Z}} 2^{j s_1 q} \|f_{j,\gamma}\|_{A_1 (\dot{\Gamma}_{j}\backslash \Omega_{j}) }^{q}\}^{\frac{1}{q}}\\
\sim &\{\sum\limits_{j\in \mathbb{Z}} [ 2^{j q s_0} \|f_{j,\gamma}\|_{A_0(\Omega_{j}) }^{q}
+ t^{q}  2^{j s_1 q} \|f_{j,\gamma}\|_{A_1 (\dot{\Gamma}_{j}\backslash \Omega_{j}) }^{q}] \}^{\frac{1}{q}}\\
\sim &\{\sum\limits_{j\in \mathbb{Z}} [ 2^{j s_0} \|f_{j,\gamma}\|_{A_0(\Omega_{j}) }
+ t \sum\limits_{j\in \mathbb{Z}} 2^{j s_1 } \|f_{j,\gamma}\|_{A_1 (\dot{\Gamma}_{j}\backslash \Omega_{j}) } ]^{q}\}^{\frac{1}{q}} .
\end{array}$$

Take infimum for the above quantities
and compare the relationship between minimum values,
we get the conclusion of the Theorem.
\end{proof}

\section{Exponential space and case $A_0\neq A_1$, $0<q_0\neq q_1<\infty$}\label{sec:6x}

We use the power function to change
the topology structure on the full grid $\dot{\Lambda}$
and the topology structure on the layer grid $\dot{\Gamma}_{j}$.
We use the power spaces twice and the commutativity between summation and minimal functional once
{\bf to reveal the triplet topology nonlinearity} of real interpolation of Besov hierarchical spaces.
Based on the commutativity of summability and minimal functional,
we obtain the specific expression of the corresponding K functional.

\subsection{Power spaces and compatibility} \label{subsec:6.1}

We introduce first vertex $K_{\infty}$ functional for psedo-norm spaces.
We denote $\Lambda_0 \bigoplus \Lambda_1= \dot{\Lambda}$, if
(i) $\Lambda_0 \bigcup \Lambda_1= \dot{\Lambda}$ and $\Lambda_0  \bigcap \Lambda_1= \phi$. Define
\begin{equation} \label{eq:11.1}
K_{\infty}(t, f, A_0, A_1)= \inf\limits_{\Lambda_0 \bigoplus \Lambda_1= \dot{\Lambda}} \max (\|f\|_{A_0(\Lambda_0)} ,  t\|f\|_{A_1(\Lambda_1)}).
\end{equation}
Similar to the above section or similar to the proof for Lemma \ref{le3.1},
it is easy to see that $K_{\infty}$ functional in the above \eqref{eq:11.1}
has the same properties as $K_{\xi}(t, f, A_0, A_1)$ functional
where $0<\xi<\infty$.

\begin{theorem}
$K_{\infty}(t, f, A_0, A_1) \sim K(t, f, A_0, A_1) \sim K_{\xi}(t, f, A_0, A_1).$
\end{theorem}

Then, to introduce topological transformations,
we need to use the inverse function of a weighted power function of a monotonically increasing function.
Let $g(t)$ be a monotonically increasing function for $0<t<\infty$.
For $0<\rho_1\leq \rho_0< \infty$ and
$$s= t^{\rho_1} g (t)^{\rho_0 -\rho_1} ,$$
there exists an function $F(g, \rho_0, \rho_1, s)$ such that
\begin{equation} \label{eq:t.g}
t= F(g, \rho_0, \rho_1, s).
\end{equation}

We choose this monotone increasing function $g(t)$ as some K functional
and we show that the $K_{\infty}$ functional between power spaces has special nonlinearity.
For $0<\rho_1\leq \rho_0< \infty, 0<t<\infty$,
we choose the monotonically increasing function as K functional $K_{\infty} (t,f, A_0, A_1)$ and consider
$$s= t^{\rho_1} K_{\infty} (t,f, A_0, A_1)^{\rho_0 -\rho_1} .$$
For fixed function $f$, by monotonicity and equation \eqref{eq:t.g}, there exists a function $F$ that describes
the nonlinearity of $t$ as a function of $s$  such that
\begin{equation} \label{eq:t.1}
t= F(f,A_0,A_1, \rho_0, \rho_1, s).
\end{equation}

For $0<\rho_0\leq \rho_1< \infty, 0<t<\infty$,
we choose the monotonically increasing function as K functional $K_{\infty} (t,f, A_1, A_0)$ and consider
$$s^{-1}= t^{\rho_0} K_{\infty} (t,f, A_1, A_0)^{\rho_1 -\rho_0} .$$
Similar to the above equation \eqref{eq:t.1}, by monotonicity and equation \eqref{eq:t.g}, there exists a function $F$ that describes
the nonlinearity of $t$ as a function of $s$  such that
\begin{equation} \label{eq:t.2}
t= F(f,A_1,A_0, \rho_1, \rho_0, s^{-1}).
\end{equation}

We consider a change in the K-functionals due to power changes in the topology of  spaces.
The real interpolation space of $A_0$ and $A_1$ has the following nonlinear relation
with the real interpolation space of their power spaces $(A_0)^{\rho_0}$ and $(A_{1})^{\rho_1}$:
\begin{lemma} \label{lem.11.1}
Given $0<\rho_0,\rho_1 < \infty, 0<s,t<\infty$.

(i) If $\rho_0\geq \rho_1$ and $t= F(f,A_0,A_1, \rho_0, \rho_1, s)$ which is defined in \eqref{eq:t.1}, then
\begin{equation}\label{eq:6.3.1}
K_{\infty}(s,f, (A_0)^{\rho_0}, (A_{1})^{\rho_1})= K_{\infty}(t,f, A_0, A_1)^{\rho_0}.
\end{equation}

(ii) If $\rho_0< \rho_1$ and $t= F(f,A_1,A_0, \rho_1, \rho_0, s^{-1})$ which is defined in \eqref{eq:t.2}, then
\begin{equation}\label{eq:6.3.2}
\begin{array}{rcl}
K_{\infty}(s,f, (A_0)^{\rho_0}, (A_{1})^{\rho_1})&=& s K_{\infty}(s^{-1},f, (A_1)^{\rho_1}, (A_{0})^{\rho_0})\\
&=& s K_{\infty}(t,f, A_0, A_1)^{\rho_1}.
\end{array}
\end{equation}

\end{lemma}

\begin{proof}
(i) We know
$$\begin{array}{l}K_{\infty}(s,f, (A_0)^{\rho_0}, (A_{1})^{\rho_1})= \inf\limits_{\Lambda_0 \bigoplus \Lambda_1= \dot{\Lambda}}
\max ( \|f\|_{A_0(\Lambda_0)}^{\rho_0}, s\|f\|_{A_1(\Lambda_1)}^{\rho_1})\\
= \inf\limits_{\Lambda_0 \bigoplus \Lambda_1= \dot{\Lambda}}
\max( (\frac{\|f\|_{A_0(\Lambda_0)}} {K_{\infty}(t,f, A_0, A_1)})^{\rho_0}, (\frac{t \|f\|_{A_1(\Lambda_1)}} {K_{\infty}(t,f, A_0, A_1))})^{\rho_1})
K_{\infty}(t,f, A_0, A_1)^{\rho_0}
\end{array}$$

Since
$$1=\inf\limits_{\Lambda_0 \bigoplus \Lambda_1= \dot{\Lambda}}
\max ( (\frac{\|f\|_{A_0(\Lambda_0)}} {K_{\infty}(t,f, A_0, A_1)})^{\rho_0}, (\frac{t \|f\|_{A_1(\Lambda_1)}} {K_{\infty}(t,f, A_0, A_1))})^{\rho_1}).$$
Which implies \eqref{eq:6.3.1}.

(ii) We know
$$\begin{array}{l}K_{\infty}(s^{-1},f, (A_1)^{\rho_1}, (A_{0})^{\rho_0})= \inf\limits_{\Lambda_0 \bigoplus \Lambda_1= \dot{\Lambda}}
\max ( \|f\|_{A_1(\Lambda_1)}^{\rho_1}, s\|f\|_{A_0(\Lambda_0)}^{\rho_0})\\
= \inf\limits_{\Lambda_0 \bigoplus \Lambda_1= \dot{\Lambda}}
\max( (\frac{\|f\|_{A_1(\Lambda_1)}} {K_{\infty}(t,f, A_1, A_0)})^{\rho_1}, (\frac{t \|f\|_{A_0(\Lambda_0)}} {K_{\infty}(t,f, A_1, A_0))})^{\rho_0})
K_{\infty}(t,f, A_1, A_0)^{\rho_1}
\end{array}$$

Since
$$1=\inf\limits_{\Lambda_0 \bigoplus \Lambda_1= \dot{\Lambda}}
\max ( (\frac{\|f\|_{A_1(\Lambda_1)}} {K_{\infty}(t,f, A_1, A_0)})^{\rho_1}, (\frac{t \|f\|_{A_0(\Lambda_0)}} {K_{\infty}(t,f, A_1, A_0))})^{\rho_0}).$$
Which implies \eqref{eq:6.3.2}.

\end{proof}

%%%%%%%%%%%%%%%
%%%%%%%%%%%%%%

Further, we introduce four exponential spaces to describe
the change of topological structure of function spaces on main grid and layer grid.
For $s_0, s_1\in \mathbb{R}, 0< q_0, q_1 < \infty$, denote
\begin{equation}
\|f\|_{X_0} = \|f\|^{q_0}_{\dot{l}^{s_0,q_0}(A_0)} \mbox{ and }
\|f\|_{X_1}= \|f\|^{q_1}_{\dot{l}^{s_1, q_1}(A_1)}.
\end{equation}
We change the topology structure of the function spaces on the full grid to get the intermediate spaces $X_0$ and $X_{1}$.
For $j\in \mathbb{Z}$,
denote
\begin{equation}
\|f\|_{Y^j_0} = \|f\|^{q_0}_{A_0 (\Gamma_j)} \mbox{ and }
\|f\|_{Y^j_1}= \|f\|^{q_1}_{A_1(\Gamma_j)}.
\end{equation}
We change the topology structure of the function spaces on the layer grid
to get the intermediate spaces $Y^j_0$ and $Y^j_1$.
We show the commutativity of summation and functional for intermediate spaces.
For $i=0,1, j\in \mathbb{Z}$, denote $\Gamma^{i}_{j}= \{\gamma\in \Gamma_{j}, (j,\gamma)\in \Lambda_i\}$.
For $s_0, s_1\in \mathbb{R}$, $0< q_0\neq q_1<\infty$,
we have
$$\begin{array}{rl}
&K_{\infty}(t, f, X_0, X_1)\\
= &\inf\limits_{\Lambda_0 \bigoplus \Lambda_1= \dot{\Lambda}} \max \{
\sum\limits_{j\in \mathbb{Z}} 2^{js_0 q_0} \|f_{j,\gamma}\|_{A_0( \Gamma^{0}_{j}) } ^{q_0} ,
t\sum\limits_{j\in \mathbb{Z}} 2^{js_1 q_1} \|f_{j,\gamma}\|_{A_1( \Gamma^{1}_{j}) } ^{q_1}\}\\
\sim &\inf\limits_{\Lambda_0 \bigoplus \Lambda_1= \dot{\Lambda}}
\{\sum\limits_{j\in \mathbb{Z}} 2^{js_0 q_0} \|f_{j,\gamma}\|_{A_0( \Gamma^{0}_{j}) } ^{q_0}
+ t\sum\limits_{j\in \mathbb{Z}} 2^{js_1 q_1} \|f_{j,\gamma}\|_{A_1( \Gamma^{1}_{j}) } ^{q_1}\}\\
\sim &
\sum\limits_{j\in \mathbb{Z}} \inf\limits_{\Gamma^{0}_j \bigoplus \Gamma^{1}_j= \Gamma_{j} }
\{2^{js_0 q_0} \|f_{j,\gamma}\|_{A_0( \Gamma^{0}_{j}) } ^{q_0}
+ t 2^{js_1 q_1} \|f_{j,\gamma}\|_{A_1( \Gamma^{1}_{j}) } ^{q_1} \}\\
\sim &
\sum\limits_{j\in \mathbb{Z}} \inf\limits_{\Gamma^{0}_j \bigoplus \Gamma^{1}_j= \Gamma_{j} }
\max \{2^{js_0 q_0} \|f_{j,\gamma}\|_{A_0( \Gamma^{0}_{j}) } ^{q_0},
t 2^{js_1 q_1} \|f_{j,\gamma}\|_{A_1( \Gamma^{1}_{j}) } ^{q_1} \}
\end{array}
$$

Hence exponential spaces of X spaces on full grid can be written as the sum of Y spaces on layer grid
and we can change the order for the sum of frequency and the limit of infimum.
\begin{lemma} \label{lem:ab}
$$\begin{array}{l}
K_{\infty}(t,f, X_0, X_1)
\sim \sum\limits_{j\in \mathbb{N}} K_{\infty}(t2^{j(s_1-s_0)}, (2^{js_0}|f_{j,\gamma}|)_{\gamma\in \Gamma_{j}}, Y^j_0, Y^j_1).
\end{array}
$$
\end{lemma}

%%%%%%%%%%%%%
%%%%%%%%%%%%%%
%%%%%%%%%%%%%

\subsection{ Topology structure for the case $0<q_0\neq q_1 <\infty$} \label{subsec:6.2}

We use the following two properties to build the topology:
(1) Commutativity of summation and functional for intermediate spaces, see Lemma \ref{lem:ab}.
(2) The topology nonlinearity between K-functional of power space and K-functional of original space.

Let's apply Lemma \ref{lem.11.1} to layer grid.
In the above Lemma \ref{lem.11.1}, we take $\rho_0=q_0, \rho_1=q_1$,
take $A_0$ to be $A_0(\dot{\Gamma}_j)$ and take $A_1$ to be $A_1(\dot{\Gamma}_j)$.
Then we have
$(A_0)^{\rho_0}= Y^j_0$ and $(A_1)^{\rho_1}= Y^j_1$.
Lemma \ref{lem.11.1} can be rewritten as following corollary for layer grid:
\begin{corollary} \label{cor:11.2}
Given $j\in \mathbb{Z}$, $0<q_0,q_1 < \infty, 0<s,t<\infty$.

(i) If $0<q_1<q_0<\infty$ and $t= F(f,A_0(\dot{\Gamma}_j),A_1(\dot{\Gamma}_j), q_0, q_1, s)$ which is defined in \eqref{eq:t.1}, then
$$K_{\infty}(s, g_j, Y^j_0, Y^j_1)= K_{\infty}(t, g_j, A_0(\dot{\Gamma}_j), A_1(\dot{\Gamma}_j))^{q_0}.$$

(ii) If $0<q_0<q_1<\infty$ and $t= F(f,A_1(\dot{\Gamma}_j),A_0(\dot{\Gamma}_j), q_1, q_0, s^{-1})$ which is defined in \eqref{eq:t.2}, then
$$K_{\infty}(s, g_j, Y^j_0, Y^j_1)= s K_{\infty}(s^{-1}, g_j, Y^j_1, Y^j_0)= s K_{\infty}(t, g_j, A_1(\dot{\Gamma}_j), A_0(\dot{\Gamma}_j))^{q_1}.$$

\end{corollary}

%%%%%%%%%%%
%%%%%%%%%%%%
%%%%%%%%%%%%%%%%

By applying Lemma \ref{lem:ab}, the K functional $K_{\infty}(t,f, X_0, X_1)$ is known.
We apply then  Lemma \ref{lem.11.1} to full grid.
In the above Lemma \ref{lem.11.1}, we take $\rho_0= \frac{1}{q_0}, \rho_1=\frac{1}{q_1}$,
take $A_0$ to be $X_0$ and take $A_1$ to be $X_1$.
Then we have
$(A_0)^{\rho_0}= \dot{l}^{s_0,q_0}(A_0)$ and $(A_1)^{\rho_1}= \dot{l}^{s_1,q_1}(A_1)$.
Lemma \ref{lem.11.1} can be rewritten as following corollary for full grid:
\begin{corollary} \label{cor:11.1}
(i) If $0<q_0<q_1<\infty$  and $t= F(f,X_0,X_1, \frac{1}{q_0}, \frac{1}{q_1}, s)$ which is defined in \eqref{eq:t.1}, then
$$K_{\infty}(s,f, \dot{l}^{s_0,q_0}(A_0), \dot{l}^{s_1, q_1}(A_1))= K_{\infty}(t,f, X_0, X_1)^{\frac{1}{q_0}}.$$

(ii) If $0<q_1<q_0<\infty$  and $t= F(f,X_1,X_0, \frac{1}{q_1}, \frac{1}{q_0}, s^{-1})$ which  is defined in \eqref{eq:t.2}, then
$$ \begin{array}{rcl}
K_{\infty}(s,f, \dot{l}^{s_0,q_0}(A_0), \dot{l}^{s_1, q_1}(A_1)) &=& s K_{\infty}(s^{-1},f, \dot{l}^{s_1,q_1}(A_1), \dot{l}^{s_0, q_0}(A_0))\\
&=& s K_{\infty}(t,f, X_1, X_0)^{\frac{1}{q_1}}.
\end{array}$$

\end{corollary}

Combined with the above techniques,
by Lemmas \ref{lem:CK} and \ref{lem:ab}, Corollaries \ref{cor:11.2} and \ref{cor:11.1},
we obtain a concrete expression for the K functional over $(\dot{l}^{s_0,q_0}(A_0), \dot{l}^{s_1,q_1}(A_1))$
characterized by a group of  K functional over layer grid $(A_0 (\dot{\Gamma}_j), A_1 (\dot{\Gamma}_j))$.

\begin{theorem} \label{th:11.4}

Given $0< q_0 \neq  q_1 < \infty$, $0<t<\infty$.

(i) For $0<q_0<q_1<\infty$,
take $\tilde{s}^{0}=s_1-s_0$
and take $g^{0}_{j}=(2^{js_0}|f_{j,\gamma}|)_{\gamma\in \Gamma_{j}}$.
Take $s= F(f,X_0,X_1, \frac{1}{q_0}, \frac{1}{q_1},t)$,
take $\tau= F(g_j, A_1(\dot{\Gamma}_j), A_0(\dot{\Gamma}_j), q_1, q_0, s^{-1}2^{-j\tilde{s}^{0}})$.
We have
$$\begin{array}{rcl}
K_{\infty}(t,f, \dot{l}^{s_0,q_0}(A_0), \dot{l}^{s_1, q_1}(A_1))
&=& \{\sum\limits_{j\in \mathbb{Z}} s2^{j\tilde{s}^{0}} K_{\infty} (\tau, g^{0}_j, A_1(\dot{\Gamma}_j), A_0(\dot{\Gamma}_j))^{q_1} \}^{\frac{1}{q_0}}.
\end{array}$$

(ii) For $0<q_1<q_0<\infty$,
take $\tilde{s}^{1}=s_0-s_1$
and take $g^{1}_{j}=(2^{js_1}|f_{j,\gamma}|)_{\gamma\in \Gamma_{j}}$.
Take $s= F(f,X_1,X_0, \frac{1}{q_1}, \frac{1}{q_0},t^{-1})$
and take $\tau= F(g_j^{1}, A_0(\dot{\Gamma}_j), A_1(\dot{\Gamma}_j), q_1, q_0, s^{-1})$.
We have
$$\begin{array}{rcl}
K_{\infty}(t,f, \dot{l}^{s_0,q_0}(A_0), \dot{l}^{s_1, q_1}(A_1))&=&
t \{\sum\limits_{j\in \mathbb{Z}} s^{2} 2^{j\tilde{s}^{1} } K_{\infty} (\tau, g_j^{1}, A_0(\dot{\Gamma}_j), A_1(\dot{\Gamma}_j))^{q_0} \}^{\frac{1}{q_1}}.
\end{array}$$

\end{theorem}

\begin{proof}

(i) If $0<q_0<q_1<\infty$, take $s= F(f,X_0,X_1, \frac{1}{q_0}, \frac{1}{q_1},t)$, then by Corollary \ref{cor:11.1},
$$\begin{array}{rcl}
K_{\infty}(t,f, \dot{l}^{s_0,q_0}(A_0), \dot{l}^{s_1, q_1}(A_1))&=& \{ K_{\infty} (s, f, X_0, X_1) \}^{\frac{1}{q_0}}.
\end{array}$$

By applying Lemma \ref{lem:ab},
$$\begin{array}{rcl}
K_{\infty}(t,f, \dot{l}^{s_0,q_0}(A_0), \dot{l}^{s_1, q_1}(A_1))&=& \{\sum\limits_{j\in \mathbb{Z}} K_{\infty} (s2^{j\tilde{s}^{0}}, g_j^{0}, Y^j_0, Y^j_1) \}^{\frac{1}{q_0}}\\
&=& \{\sum\limits_{j\in \mathbb{Z}} s2^{j\tilde{s}^{0}} K_{\infty} (s^{-1}2^{-j\tilde{s}^{0}}, g_j^{0}, Y^j_1, Y^j_0) \}^{\frac{1}{q_0}}.
\end{array}$$

Take $\tau= F(g_j^{0}, A_1(\dot{\Gamma}_j), A_0(\dot{\Gamma}_j), q_1, q_0, s^{-1}2^{-j\tilde{s}^{0}})$, then by Corollary \ref{cor:11.2},
$$\begin{array}{rcl}
K_{\infty}(t,f, \dot{l}^{s_0,q_0}(A_0), \dot{l}^{s_1, q_1}(A_1))
&=& \{\sum\limits_{j\in \mathbb{Z}} s2^{j\tilde{s}^{0}} K_{\infty} (\tau, g_j^{0}, A_1(\dot{\Gamma}_j), A_0(\dot{\Gamma}_j))^{q_1} \}^{\frac{1}{q_0}}.
\end{array}$$

(ii) If $0<q_1<q_0<\infty$, take $s= F(f,X_1,X_0, \frac{1}{q_1}, \frac{1}{q_0},t^{-1})$, then by Corollary \ref{cor:11.1},
$$\begin{array}{rcl}
K_{\infty}(t,f, \dot{l}^{s_0,q_0}(A_0), \dot{l}^{s_1, q_1}(A_1))&=&
tK_{\infty}(t^{-1},f, \dot{l}^{s_1,q_1}(A_1), \dot{l}^{s_0, q_0}(A_0))\\
&=& t K_{\infty} (s,f, X_1, X_0) ^{\frac{1}{q_1}}.
\end{array}$$

By applying Lemma \ref{lem:ab},
$$\begin{array}{rcl}
K_{\infty}(t,f, \dot{l}^{s_0,q_0}(A_0), \dot{l}^{s_1, q_1}(A_1))&=&
t \{\sum\limits_{j\in \mathbb{Z}} s2^{j\tilde{s}^{1}} K_{\infty} (s, g_j^{1}, Y^j_1, Y^j_0) \} ^{\frac{1}{q_1}}\\
&=&
t \{\sum\limits_{j\in \mathbb{Z}} s^{2} 2^{j\tilde{s}^{1}} K_{\infty} (s^{-1}, g_j^{1}, Y^j_0, Y^j_1) \} ^{\frac{1}{q_1}}.
\end{array}$$

Take $\tau= F(g_j^{1}, A_0(\dot{\Gamma}_j), A_1(\dot{\Gamma}_j), q_1, q_0, s^{-1})$, then by Corollary \ref{cor:11.2},
$$\begin{array}{rcl}
K_{\infty}(t,f, \dot{l}^{s_0,q_0}(A_0), \dot{l}^{s_1, q_1}(A_1))
&=& t \{\sum\limits_{j\in \mathbb{Z}} s^{2} 2^{j\tilde{s}^{1} } K_{\infty} (\tau, g_j^{1}, A_0(\dot{\Gamma}_j), A_1(\dot{\Gamma}_j))^{q_0} \}^{\frac{1}{q_1}}.
\end{array}$$

\end{proof}

\section{Case $q_0\neq q_1$ and $q_0 q_1=\infty$} \label{sec:77}

We use conditional G functionals and nonlinearity between thresholds of conditional functionals
to establish topology structure of K functional.
The following Example \ref{ex:m} shows that the relative K functional for the cases $q_0\neq q_1$ and $q_0 q_1=\infty$
has different functional structure and different topology structure than which for the cases $0<q_0\neq q_1<\infty$,
we could not use the above exponential space skills in the above Section.
\begin{example}\label{ex:m}
For any integer $m>0$,
there exists set $S\subset \mathbb{Z}$ such that $\sharp S=m$
and $\forall j,j'\in S$, $f_{j}$ and $f_{j'}$ satisfies the following condition:
$$2^{js_0} K_{\infty}(t2^{j (s_1-s_0)}, f_{j}, A_0, A_1)= 2^{j's_0} K_{\infty}(t2^{j' (s_1-s_0)}, f_{j'}, A_0, A_1).$$
By the arbitrariness of numbers $m$, we know that
the expression of the K functional is not limited to the nonlinearity of each single layer.
\end{example}

To avoid the situation of Example \ref{ex:m}, we use a different topology to handle this situation.
We introduce conditional functionals $G$ to compute the corresponding K functionals.
By the inclusion relation of function spaces, we have
\begin{lemma}\label{lem:7.2}
$$\begin{array}{rcl}
K(t,f, \dot{l}^{s_0,\infty}(A_0), \dot{l}^{s_1,\infty}(A_1)) &\leq & K_{\infty}(t,f, \dot{l}^{s_0,q_0} (A_0), \dot{l}^{s_1, \infty} (A_1))\\
&\leq & K(t,f, \dot{l}^{s_0,q_0}(A_0), \dot{l}^{s_1,q_0}(A_1)).
\end{array}$$
\end{lemma}
By applying Lemmas \ref{lem:commutative} and \ref{lem:7.2},
\begin{theorem} \label{th:4}
For $s_0, s_1\in \mathbb{R}, 0<q_0, q_1<\infty$,
there exists two monotone non-increasing functions
$\tilde{G}(s,f)$ and $H(s,f)$ such that
\begin{equation} \label{eq:7.1}
\begin{array}{l}K_{\infty}(t,f, \dot{l}^{s_0,q_0} (A_0), \dot{l}^{s_1, \infty} (A_1))= H(t,f) \tilde{G}(H(t,f), f).
\end{array}
\end{equation}
\begin{equation} \label{eq:7.2}
\begin{array}{l}K_{\infty}(t,f, \dot{l}^{s_0,\infty} (A_0), \dot{l}^{s_1, q_1} (A_1))
= t K_{\infty}(t^{-1},f, \dot{l}^{s_1, q_1} (A_1), \dot{l}^{s_0,\infty} (A_0)).
\end{array}
\end{equation}
\end{theorem}

\begin{proof}
We introduce conditional functional $G$ on full grid which is a monotone non-increasing function
\begin{equation}\label{eq:con.fun}
G(s,f, A_0, A_1)= \inf\limits_{\Lambda_0 + \Lambda_1= \dot{\Lambda}, \|f\|_{A_1(\Lambda_1)}\leq s}   \|f\|_{A_0(\Lambda_0)}.
\end{equation}
For $s>0, q_1=\infty$, for function $\{f_{j,\gamma}\}_{\gamma\in \dot{\Gamma}_{j}}$ on layer grid,
the following conditional K functional can be defined by the conditional K functional $G(s,f, A_0, A_1)$:
\begin{equation*}
\begin{array}{rl}
&G(s,\{f_{j,\gamma}\}_{\gamma\in \dot{\Gamma}_{j}}, \dot{l}^{s_0,q_0} (A_0), \dot{l}^{s_1, \infty} (A_1))\\
=& \inf\limits_{\Lambda_0 + \Lambda_1= \dot{\Lambda}, 2^{j s_1}\|\{f_{j,\gamma}\}_{\gamma\in \dot{\Gamma}_{j}}\|_{A_1(\Lambda_1)}\leq s}
2^{js_0} \|\{f_{j,\gamma}\}_{\gamma\in \dot{\Gamma}_{j}}\|_{A_0(\Lambda_0)}.
\end{array}
\end{equation*}
In the case of conditional functionals,
a functional on a full grid has the following relation to a functional on a layer grid.
$$G(s,f)=G(s,f, \dot{l}^{s_0,q_0} (A_0), \dot{l}^{s_1, \infty} (A_1))
=\{ \sum\limits_{j\in \mathbb{Z}} G(s,\{f_{j,\gamma}\}_{\gamma\in \dot{\Gamma}_{j}}, \dot{l}^{s_0,q_0} (A_0), \dot{l}^{s_1, \infty} (A_1))^{q_0}\}^{\frac{1}{q_0}}.$$

The next step is to convert the constraint condition
$\|f\|_{A_1(\Lambda_1)}\leq s$ or $\|\{f_{j,\gamma}\}_{\gamma\in \dot{\Gamma}_{j}}\|_{A_1(\Lambda_1)}\leq s$
into something related to $t$.
We know
\begin{equation}\label{eq:7.44}
\tilde{G}(t,f)=t^{-1} G(t,f) {\mbox \rm \,is\, a\, monotone\, non\, increasing\, function.}
\end{equation}
Hence by applying Lemma \ref{lem:7.2},
there exists $s$ satisfying that
$$K(t,f, \dot{l}^{s_0,\infty}(A_0), \dot{l}^{s_1,\infty}(A_1)) \leq ts\leq  K(t,f, \dot{l}^{s_0,q_0}(A_0), \dot{l}^{s_1,q_0}(A_1))$$
such that $\tilde{G}(s,f)=t$.
By applying Equation \eqref{eq:7.44}, denote
\begin{equation}\label{eq:H}
s= H(t,f)
\end{equation}
We get equation \eqref{eq:7.1}.

For $q_0=\infty$, by Lemma \ref{lem:CK} and equation \eqref{eq:7.1},
we get equation \eqref{eq:7.2} and the conclusion of Theorem \ref{th:4}.

\end{proof}

%%%%%%%%%%%%%%%%%%%%%%%%%%%%%%%%%%%%%%%%%%%%%%%%%%%%%%%%%%%%%%%%%%%%

\textbf{Acknowledgements.}
Authors of this paper would like to thank Professor Hans Triebel for his useful discussion,
his valuable suggestions, and for providing us with some references.
The first author would like to thank Professor Yan's invitation to University of Chinese Academy of Sciences and his hospitality.
The final version of the full text has been discussed in detail in the Harmonic Analysis Group during this visit.
The authors would like to thank all members of the group for their valuable opinions,
especially for some improvements to the proof details of the manuscript from Professor Dunyan Yan and Doctor Boning Di.

%{\color{red} This project is partially supported by the National Natural Science Foundation of China (Grant Nos. 12071052, 12271501).}

%\bigskip
%\noindent Hans Triebel

%\medskip

%\noindent Institute of Mathematics, Friedrich Schiller University Jena, \\
%Jena, 07737, Germany
%\smallskip

%\noindent{\it E-mail address}: \texttt{hans.triebel@uni-jena.de}

\bigskip

\noindent Qixiang Yang

\medskip

\noindent School of Mathematics and Statistics, Wuhan University. \\
Wuhan, 430072, China
\smallskip

\noindent{\it E-mail address}:
\texttt{qxyang@whu.edu.cn}

\bigskip
\noindent Haibo Yang

\medskip
\noindent
Macau Institute of Systems Engineering, \\
Macau University of Science and Technology, Macau, 999078, China.\\
\smallskip
\noindent{\it E-mail address}:
\texttt{yanghb97@qq.com}

%\bigskip
%\noindent Boning Di

%\medskip

%\noindent School of Mathematical Sciences, University of Chinese Academy of Sciences \\
%Beijing, 100049, China
%\smallskip

%\noindent{\it E-mail address}:
%\texttt{diboning18@mails.ucas.ac.cn}

%\bigskip
%\noindent Dunyan Yan

%\medskip

%\noindent School of Mathematical Sciences, University of Chinese Academy of Sciences \\
%Beijing, 100049, China
%\smallskip

%\noindent{\it E-mail address}:
%\texttt{ydunyan@ucas.ac.cn}

\end{document}